\documentclass[11pt,a4paper]{amsart}
\usepackage{amsmath,amsfonts,amsthm,amsopn,color,amssymb,enumitem}
\usepackage{palatino}
\usepackage{graphicx}
\usepackage{caption}
\usepackage{subcaption}
\usepackage[colorlinks=true]{hyperref}
\hypersetup{urlcolor=blue, citecolor=red, linkcolor=blue}

\usepackage{esint}
\usepackage[title]{appendix}

\numberwithin{equation}{section}

\newtheorem{theorem}{Theorem}[section]
\newtheorem{lemma}[theorem]{Lemma}

\newtheorem{proposition}[theorem]{Proposition}
\newtheorem{remark}[theorem]{Remark}

\newcommand{\e}{\varepsilon}

\newcommand{\R}{\mathbb{R}}

\newcommand{\D}{\mathbb{D}}

\newcommand{\weakto}{\rightharpoonup}

\DeclareMathOperator{\supp}{supp}

\renewcommand{\b }{\beta }

\newcommand{\N}{\mathbb{N}}

\def\bbm[#1]{\mbox{\boldmath $#1$}}
\newcommand{\beq }{\begin{equation}}
	\newcommand{\eeq }{\end{equation}}

\def\sideremark#1{\ifvmode\leavevmode\fi\vadjust{\vbox to0pt{\vss
			\hbox to 0pt{\hskip\hsize\hskip1em
				\vbox{\hsize3cm\tiny\raggedright\pretolerance10000
					\noindent #1\hfill}\hss}\vbox to8pt{\vfil}\vss}}}%

\begin{document}
	
	\title[Prescribing curvatures in the disk]{Prescribing curvatures in the disk via conformal changes of the metric: the case of negative Gaussian curvature}

\author {Rafael L\'opez-Soriano, Francisco J. Reyes-S\'anchez, David Ruiz}

\address{Universidad de Granada \\
	IMAG, Departamento de An\'alisis Matem\'atico \\
		Campus Fuentenueva 	\\	18071 Granada, Spain. }

\

 \email{ralopezs@ugr.es, fjreyes@ugr.es, daruiz@ugr.es}	
	
%
%
	
	\thanks{R. L.-S. and D. R. has been supported by the  MICIN/AEI through
		the Grant PID2021-122122NB-I00 and the \emph{IMAG-Maria de Maeztu} Excellence Grant CEX2020-001105-M, and by J. Andalucia via the Research Group FQM-116.  R.L.-S. has also been supported by the grant Juan de la Cierva Incorporaci\'on fellowship (JC2020-046123-I), funded by MCIN/AEI/10.13039/501100011033, and by the European Union Next Generation EU/PRTR. F.J.R.S has  been supported by a PhD fellowship (PRE2021-099898) linked to the \emph{IMAG-Maria de Maeztu} Excellence Grant CEX2020-001105-M funded by MICIN/AEI.} 	

	
	
	\begin{abstract}
		This paper deals with the question of prescribing the Gaussian curvature on a disk and the geodesic curvature of its boundary by means of a conformal deformation of the metric. We restrict ourselves to a symmetric setting in which the Gaussian curvature is negative, and we are able to give general existence results. Our approach is variational, and solutions will be searched as critical points of an associated functional. The proofs use a perturbation argument via the monotonicity trick of Struwe, together with a blow-up analysis and Morse index estimates. We also give a nonexistence result that shows that, to some extent, the assumptions required for existence are necessary.
	\end{abstract}
	
	\maketitle
	
	\section{Introduction}\label{section-introduction}
Let $(\Sigma,g_0)$ denote a compact Riemannian surface $\Sigma$ equipped with a certain metric $g_0$. We say that another metric $g$ is conformal to $g_0$ if $ {g}= \rho \, g_0$, where $\rho: \Sigma \to (0, +\infty)$ is the conformal factor.
	A classical problem in Geometry, dating back to \cite{Berger1971,KazdanWarner1974}, is the prescription of the Gaussian curvature on a compact Riemannian surface under a conformal change of
	the metric. In this case if we take the conformal factor as $\rho=e^u$, the curvature transforms according to the law:

	\begin{equation*}
		-\Delta u+2 K_0(x)=2K(x)e^{u}.
	\end{equation*}
	Here $\Delta$ stands for the Laplace-Beltrami operator associated to the metric $g_0$, and $K_0$, ${K}$ denote the Gaussian curvatures with respect to $g_0$ and $g$, respectively. The solvability of this equation has been studied for several decades, and it is not possible to give here a comprehensive list of references. The case $\Sigma = \mathbb{S}^2$ is known as the Nirenberg problem, and reveals to be particularly delicate. We refer the interested reader to Chapter 6 in the book \cite{Aubin1998}.
	
	If $\Sigma$ has a boundary, other than the Gaussian curvature in $\Sigma$ it is natural to prescribe also the geodesic curvature  of $\partial\Sigma$. Denoting by  $h_0$ and ${h}$ the geodesic curvatures of the boundary with respect to $g_0$ and ${g}$ respectively, we are led to the boundary  value problem:

	\begin{equation}\label{sigma}
		\left\{
		\begin{array}{ll}
			-\Delta u+2 K_0(x)=2K(x)e^{u}&\text{in }\Sigma,\\
			\\
			\frac{\partial u}{\partial \nu}+2 {h}_0(x)=2h(x)e^{u/2}&\text{on }\partial\Sigma.\\
		\end{array}\right.
	\end{equation}

	The case of constant curvatures $K$, $h$ has been considered in \cite{Brendle2002}, where the author used a parabolic flow to obtain solutions in the limit. Some classification results for the half-plane are also available in \cite{GalvezMira2009, LiZhu1995, Zhang2003}. The case of nonconstant curvatures was first addressed in \cite{Cherrier1984}, see also \cite{BaoWangZhou2014, LopezSorianoMalchiodiRuiz2019} for related work. In particular, the case $K<0$ has been relatively well understood in \cite{LopezSorianoMalchiodiRuiz2019} for non \textbf{}simply connected surfaces. Recently, some existence results have been provided for all genuses by using a new mean-field  approach in \cite{BattagliaLS2023}. The blow-up analysis of solutions in this framework has also been studied, see \cite{CajuCruzSantos2024, LopezSorianoMalchiodiRuiz2019} for surfaces with negative Euler characteristic.

	If $\Sigma=\mathbb{D}$, by the classical Uniformization Theorem we can pass to the standard metric via a conformal change of the metric. Hence we can assume that $g_0$ is the standard metric on the disk, and equation \eqref{sigma} becomes:
	
	\begin{equation}
		\left\{
		\begin{array}{ll}
			-\Delta u=2K(x)e^{u}&\text{in }\mathbb{D},\\
			\\
			\frac{\partial u}{\partial \nu}+2=2h(x)e^{u/2}&\text{on }\partial\mathbb{D}.\\
		\end{array}\right.
		\label{eq: Problem 1}
	\end{equation}
	
	Integrating (\ref{eq: Problem 1}) we obtain the following identity, which is nothing but the Gauss-Bonnet theorem for the metric ${g}$:
	
	\begin{equation}
		\int_{\mathbb{D}}Ke^u+\int_{\partial\mathbb{D}}he^{u/2}=2\pi.
		\label{eq: Gauss-Bonnet}
	\end{equation}
	
	In particular, if $K<0$, this forces $h$ to be positive somewhere.
	
	The problem \eqref{eq: Problem 1} has been treated in several works. For example, the case $h=0$ and $\partial_{\nu}K=0$ on $\partial \D$ has been treated in \cite{ChangYang1988} passing to a Nirenberg problem via reflection (see also \cite{GuoHu1995}). The case $K=0$ has attracted more attention, see \cite{ChangLiu1996,  GuoLiu2006,  LiLiu2005, DaLioMartinazziRiviere2015, LiuHu2005}. A flow approach has been applied in \cite{Gehrig2020} to prove the existence of solutions with constant geodesic curvature. Recently, this approach has been generalized to the case of positive and nonconstant curvatures, \cite{Struwe2024}.
	
	However, very few existence results are available on equation \eqref{eq: Problem 1} when both curvatures are not constant. If both curvatures are positive the existence has been addressed in \cite{CruzRuiz2018} (under symmetry assumptions) and more recently in \cite{Ruiz2023}. We would like to stress that if $K \geq 0$, $h \geq 0$, then the total area and length:
	
	$$ \int_{\mathbb{D}}Ke^u, \ \int_{\partial\mathbb{D}}he^{u/2},$$
	are a priori bounded by the Gauss-Bonnet formula \eqref{eq: Gauss-Bonnet}. Moreover, as it will be commented later, the existence of solutions is linked to the blow-up analysis of \eqref{eq: Problem 1}. If both curvatures are nonnegative, one can use the blow-up analysis of \cite{JevnikarLopezSorianoMedinaRuiz2022}: the blow-up point is unique, located at $\partial \D$, and the asymptotic profile is well understood. On the other hand, the existence of blowing-up solutions has been recently proved (see  \cite{BattagliaCruzBlazquezPistoia2023,BattagliaMedinaPistoia2021, BattagliaMedinaPistoia2023}) via singular perturbation methods.
	
	In this paper we plan to give general existence results to (\ref{eq: Problem 1}) when $K(x)<0$. This is probably the most intricate scenario: first, because the Gauss-Bonnet does not give bounds on area and length. Second, because for $K<0$ the blow-up analysis is much more complex, different types of blow-up points could appear and the area and length could be unbounded, see \cite{LopezSorianoMalchiodiRuiz2019}. For these reasons, in this paper we restrict ourselves to a symmetric setting, as a first step in this direction.
	
	The problem (\ref{eq: Problem 1}) is the Euler-Lagrange equation of the energy functional $ \mathcal{I}:H^1(\mathbb{D})\rightarrow\mathbb{R}$,
	
	\begin{equation}\label{eq: Energy Functional Disk}
		\mathcal{I}(u)=\int_{\mathbb{D}}\left(\frac{1}{2}|\nabla u|^2-2K(x)e^
		u\right)+\int_{\partial\mathbb{D}}\left(2u-4h(x)e^{u/2}\right).
	\end{equation}
	
	Note that the interior term and the boundary one might be in competition, and a priori it is not clear whether $\mathcal{I}$ is bounded from below or not. This phenomenon is studied in \cite{LopezSorianoMalchiodiRuiz2019}, where the following scale-invariant function is introduced:
	\begin{equation*}
	\mathfrak{D}:\partial\mathbb{D}\longrightarrow\mathbb{R}, \ \ 	\mathfrak{D}(x)=\frac{h(x)}{\sqrt{|K(x)|}}.
	\end{equation*}
	
	This function has an important role in the behavior of the functional $\mathcal{I}$, as well as in the blow-up analysis of solutions to (\ref{eq: Problem 1}).
		
	We now set the following assumptions:
    
	\begin{equation}\tag{$H_1$}\label{eq: hypothesis_1}
	K:\mathbb{D}\rightarrow\mathbb{R} \text{ is a }\mathcal{C}^{1}\text{ function such that }K\leq0\text{ in }\D\text{ and }K<0\text{ on }\partial\D,
	\end{equation}
	
	\begin{equation}\tag{$H_2$}\label{eq: hypothesis_2}
		h:\partial\mathbb{D}\rightarrow\mathbb{R} \text{ is a }\mathcal{C}^{1}\text{ function such that }\mathfrak{D}(p)>1\text{ for some }p\in\partial\mathbb{D}\text{, and}
	\end{equation}
	
	\begin{equation}\tag{$H_3$}\label{eq: hypothesis_3}
		\mathfrak{D}_\tau(p)\neq0\text{ for any }p\in\partial\mathbb{D}\text{ such that }\mathfrak{D}(p)=1.
	\end{equation}
	
	Notice that $\mathfrak{D}_{\tau}$ means the derivative of $\mathfrak{D}$ with respect to the tangential direction $\tau$ along $\partial\mathbb{D}$. 
	
We also assume certain symmetry for both curvatures. Two types of symmetry groups $G$ will be considered. The first one is the group generated by $g_k$ the rotation of angle $2\pi/k$, $k \in \mathbb{N}$, $k \geq 2$. The second one is the radially symmetric case $G=O(2)$. We say that a function $f$ is $G$-symmetric if $f(x)=(f\circ g)(x)$ for all $g\in G$ and for all $x$ in the domain of $f$. Then, we assume that:

	\begin{equation}\tag{$G$}\label{eq: hypothesis_symmetry}
		K \text{ and } h \text{ are } G\text{-symmetric, where either } G = \langle g_k \rangle \text{ or } G=O(2).
	\end{equation}
	
	Let us point out that if $G=O(2)$ then $\mathfrak{D}$ is constant and hence (\ref{eq: hypothesis_2}) implies that (\ref{eq: hypothesis_3}) is trivially satisfied.
	
	Our main result in this paper is the following:
	
	\begin{theorem}\label{thm: existence}
		Under hypotheses (\ref{eq: hypothesis_1}), (\ref{eq: hypothesis_2}), (\ref{eq: hypothesis_3}) and (\ref{eq: hypothesis_symmetry}), the problem (\ref{eq: Problem 1}) admits a $G$-symmetric solution.
	\end{theorem}
	
	We will show that under hypotheses (\ref{eq: hypothesis_1}), (\ref{eq: hypothesis_2}) the functional $\mathcal{I}$ is not bounded from below, and it exhibits a mountain-pass geometry. However the Palais-Smale property is not known to hold for the functional $\mathcal{I}$. 	We bypass this problem via the well-known \textit{monotonicity trick} of Struwe, combined with a Morse index bound, as in \cite{LopezSorianoMalchiodiRuiz2019}. In this paper we take profit of the abstract setting given in \cite[Section 4]{BellazziniRuiz23}, also in \cite{BorthwickChangJeanjeanSoave-p}. By doing so, we are able to perturb the problem \eqref{eq: Problem 1} and obtain solutions to such approximating problems with an uniform Morse index bound.
	
	The next step of the proof is to show convergence of the approximating solutions $u_n$, which is equivalent to show uniform boundedness from above. Reasoning by contradiction, we assume that $u_n$ is a blowing-up sequence of solutions. At this point, our proofs for the case $G= \langle g_k \rangle$ and $G=O(2)$ are different.
	
	If $G= \langle g_k \rangle$ we make use of the blow-up analysis peformed in \cite{LopezSorianoMalchiodiRuiz2019}. In sum, there may be two kind of blow-up points: those related to bubbles and those related to $1D$ asymptotic profiles. The second type is excluded by assumption (\ref{eq: hypothesis_3}), and the first type is forbidden thanks to (\ref{eq: hypothesis_symmetry}). However, in order to apply this blow-up analysis, one needs a uniform bound on the Morse index of the approximating solutions. Observe that our variational argument implies a uniform bound only on the $G$-symmetric Morse index. Hence we need to come back to the proof of \cite{LopezSorianoMalchiodiRuiz2019} and adapt it to our $G$-symmetric setting. 
	
	The case $G=O(2)$ is simpler and the blow-up analysis of \cite{LopezSorianoMalchiodiRuiz2019} is not needed. We are able to exclude blow-up in this case by usual ODE techniques. Quite remarkably, the assumption $\mathfrak{D}>1$ appears naturally in this computation.

	Finally, we prove the following result of nonexistence of solutions that shows that assumption \eqref{eq: hypothesis_2} is necessary.
	
	\begin{theorem}\label{thm: nonexistence} 
		Let $K:\mathbb{D}\rightarrow\mathbb{R}$ be a continuous function and a H\"{o}lder continuous function $h:\partial\mathbb{D}\rightarrow\mathbb{R}$. Assume that there exists $c_0>0$ such that $K(x)\leq-c^{2}_{0}$ for all $x\in\mathbb{D}$ and $h(x)\leq c_0$ for all $x\in\partial\mathbb{D}$. Then, the problem \eqref{eq: Problem 1} is not solvable.
	\end{theorem}
	
	The proof of Theorem \ref{thm: nonexistence} relies on a sub-super solution method, first, to pass to the case $K=-1$. Then we use the classification of Liouville for the solutions of the problem:
	
	$$ - \Delta u = -2 e^u \qquad \mbox{ in } \D.$$
	In this way, we can immerse via a local isometry $(\D, e^u dz)$ as a subdomain of the hyperbolic space $\mathbb{H}$. A typical contact argument with disks in $\mathbb{H}$ concludes the proof.
	
	\medskip 
	The rest of the paper is organized as follows. In Section 2 we use a variational argument to find approximate solutions to the problem. A blow-up analysis of such solutions is performed in Section 3, which implies compactness and concludes the proof of Theorem \ref{thm: existence}. Section 4 is devoted to the proof of Theorem \ref{thm: nonexistence}. In the final Appendix we evaluate the functional $\mathcal{I}$ on suitably defined test functions, showing that this functional is unbounded from below.
	
	\bigskip	
	
	\textbf{Notation:} 	We shall use the symbols $o(1)$, $O(1)$ in a standard way to denote quantities that converge to 0 or are bounded, respectively. Analogously, we will write $o(\rho)$, $O(\rho)$ to denote quantities that, divided by $\rho$, converge to $0$ or are bounded, respectively.
	
	Moreover, throughout this paper we denote $\mathcal{C}^{1,\alpha}_{G}(\D) \subset \mathcal{C}^{1,\alpha}(\D)$ the subspace of $G$-symmetric functions. Analogously, we  denote the Sobolev space of $G$-symmetric functions by $H^1_G(\D)$.

\

\textbf{Acknowledgements:} The authors would like to express their gratitude to the anonymous referee for their valuable suggestions and comments.	
	
	
	\section{Preliminaries and approximate solutions} \label{sec:Existence-result}
	
	In this section we perform the variational study of the energy functional $\mathcal{I}$ defined in (\ref{eq: Energy Functional Disk}). In this regard, a useful tool in the analysis will be the well-known Moser-Trudinger type inequality. Let us recall a weak version for the boundary term on the disk, which is called the Lebedev-Milin inequality (see \cite{OsgoodPhillipsSarnak1988}):
	
	\begin{equation}\label{eq:MT}
		16\pi\log{\int_{\partial\D} e^{u/2}}\leq \int_{\D}|\nabla u|^2+ 4\int_{\partial\D}u.
	\end{equation}
	
	We begin by showing that $\mathcal{I}$ admits a min-max geometry.
	\begin{proposition}\label{prop: Mountain-pass}
		Assume (\ref{eq: hypothesis_1}), (\ref{eq: hypothesis_2}) and \eqref{eq: hypothesis_symmetry}, then there exist $u_0, u_1\in H^1_G(\mathbb{D})$ such that
		\begin{equation*}
			c=\inf_{\gamma\in\Gamma}\max_{t\in[0,1]}\mathcal{I}\left(\gamma(t)\right)>\max\left\{\mathcal{I}(u_0),\mathcal{I}(u_1)\right\},
		\end{equation*}
		where $\Gamma=\left\{\gamma:[0,1]\rightarrow H_G^1(\mathbb{D})\mbox{ continuous } : \gamma(0)=u_0, \gamma(1)=u_1\right\}$.
	\end{proposition}
	
	The proof of Proposition \ref{prop: Mountain-pass} requires two preliminary results:
	
	\begin{lemma}\label{lem: legofmontain}
		Assume that (\ref{eq: hypothesis_1}) is satisfied, and take $\delta>0$. Then $\mathcal{I}|_{M_\delta}$ is bounded from below, where:
		
		\begin{equation*}
			M_{\delta}:=\left\{u\in H^1(\mathbb{D}):\int_{\partial\mathbb{D}}e^{u/2} = \delta \right\}.
		\end{equation*}
		
	\end{lemma}
	\begin{proof}
		
		
		The proof is a direct consequence of the inequality \eqref{eq:MT}, since:
		
		\begin{equation*}
			\begin{split}
				\mathcal{I}(u)\geq&8\pi\log{\int_{\partial\mathbb{D}}e^{u/2}}-\int_{\partial\mathbb{D}}4h(x)e^{u/2}\\
				\geq&8\pi\log{\delta}-4\delta\mathcal{H},
			\end{split}
		\end{equation*}
		where $\displaystyle{\mathcal{H}=\max_{\partial\mathbb{D}}{h(x)}}$.
	\end{proof}
	
	\begin{lemma}\label{lem: bubble}
		If (\ref{eq: hypothesis_1}), (\ref{eq: hypothesis_2}) and \eqref{eq: hypothesis_symmetry} are satisfied, then  there exists a sequence $u_n\in H^{1}_G(\mathbb{D})$ such that 
		\begin{equation*}\label{eq:asymptotic}
			\mathcal{I}(u_n)\rightarrow-\infty\text{ and }\int_{\partial\mathbb{D}}e^{u_n/2}\rightarrow+\infty.
		\end{equation*}
	\end{lemma}
	
	\begin{proof}
		The proof of this result is postponed to the Appendix~\ref{sec: appendix}.
	\end{proof}
	
	\begin{proof}[\underline{Proof of Proposition \ref{prop: Mountain-pass}}]
		Let us evaluate the functional on constant functions $-n$, with $n \in \N$:
		
		\begin{equation*}
			\mathcal{I}(-n)=-\int_{\mathbb{D}}2K(x)e^{-n}-\int_{\partial\mathbb{D}}2 n-\int_{\partial\mathbb{D}}4h(x)e^{-n/2} \to - \infty \mbox{ as } n \to +\infty.
		\end{equation*}
		Fix $\delta >0$ and take $u_0=-n$ such that:
		$$ \int_{\partial \D} e^{u_0/2} < \delta \quad \mbox{ and }\quad \mathcal{I}(u_0) < \inf_{M_\delta} \, \mathcal{I}.$$
		
		By Lemma \ref{lem: bubble}, there exists also $u_1$ such that: $$\int_{\partial\mathbb{D}} e^{u_1/2}>\delta  \quad \mbox{ and } \quad \mathcal{I}(u_1)<\inf_{M_\delta} \, \mathcal{I}.$$ Observe that for any $\gamma\in\Gamma$ there exists $t\in(0,1)$ such that $\int_{\partial\mathbb{D}}e^{\gamma(t)/2}=\delta$. As a consequence, 
		
		\begin{equation*}
			c:=\inf_{\gamma\in\Gamma}\max_{t\in[0,1]}\mathcal{I}\left(\gamma(t)\right)\geq \inf_{M_\delta} \, \mathcal{I} >\max\{\mathcal{I}(u_0),\mathcal{I}(u_1)\}.
		\end{equation*}
		
	\end{proof}
	
	The main difficulty to prove Theorem \ref{thm: existence} is the fact that we do not know whether the Palais-Smale sequences are bounded in this framework. We shall use the so-called monotonicity trick (see \cite{Struwe1985, Jeanjean1999}), combined with Morse index bounds, to bypass this difficulty. This is an approach previously followed in \cite{LopezSorianoMalchiodiRuiz2019} in the search of conformal metrics with prescribed curvatures, and also in \cite{BellazziniRuiz23, BorthwickChangJeanjeanSoave2023} in other contexts. An abstract setting of the method has been given in \cite[Section 4]{BellazziniRuiz23}, see also \cite{BorthwickChangJeanjeanSoave-p} for a version applicable to constrained functionals. 
	
	The goal is to find solutions of approximate problems:

	\begin{equation}
		\left\{
		\begin{array}{ll}
			-\Delta u_n+2\tilde{K}_n(x)=2K_n(x)e^{u_n} &\text{in }\mathbb{D},\\
			\\
			\frac{\partial u_n}{\partial \nu}+2\tilde{h}_n(x)=2h_n(x)e^{u_n/2} &\text{on }\partial\mathbb{D}.\\
		\end{array}\right.
		\label{eq: perturbed problem 1}
	\end{equation}
	
	These are critical points of the perturbed functional $\mathcal{I}_n: H^1_G(\D) \to \R$, 
	
	$$\mathcal{I}_n(u) = \int_{\mathbb{D}}\left(\frac{1}{2}|\nabla u|^2+2\tilde{K}_n(x)u-2K_n(x)e^{u}\right)+\int_{\partial\mathbb{D}}\left(2\tilde{h}_n(x)u-4h_n(x)e^{u/2}\right).$$
	
	As usual, we define the Morse index of the solutions $u_n$ as:
	
	\begin{equation*} \label{morse} ind_G(u_n):= \max \{\dim \ E; E \subset H^1_G(\D) \mbox{ vector space}: \mathcal{I}_n''(u_n)|_{E\times E} \mbox{ is negative definite} \}. \end{equation*}

	Observe that this is nothing but the number of negative eigenvalues of $\mathcal{I}_n''(u_n)$, counted with multiplicity, considering only $G$-symmetric eigenfunctions.
	
	For later use, we define the quadratic form:
	
	\begin{equation}\label{eq:quadraticform}
		Q_n(\psi):=\mathcal{I}_n''(u_n)(\psi,\psi)=\int_{\mathbb{D}}\left(|\nabla\psi|^{2}-2\psi^{2}K_n(x)e^{u_n}\right)-\int_{\partial\mathbb{D}}\psi^{2}h_n(x)e^{u_n/2}.
	\end{equation}

	The main result of this section is the following:
	
	\begin{proposition}\label{prop:ConvergenceSubsequence}
		Under the assumptions of Theorem \ref{thm: existence}, there exists a sequence of solutions $u_n$ to problem (\ref{eq: perturbed problem 1}), where $\tilde{K}_n(x)\rightarrow 0$, $\tilde{h}_n(x)\rightarrow 1$, $K_n(x)\rightarrow K(x) $, $h_n(x)\rightarrow h(x)$ in the $\mathcal{C}^1$ sense. Moreover,
		\begin{equation}\label{morse<1}
			ind_G(u_n) \leq 1.
		\end{equation} 		
	\end{proposition}

	\begin{proof}
		
		Let us consider the following family of functionals
		
		\begin{equation*}
			\mathcal{I}_{\varepsilon}(u)=\mathcal{I}(u)+\varepsilon \mathcal{J}(u),
		\end{equation*}
		where $\mathcal{J}:H^1(\mathbb{D})\to\mathbb{R}$ is defined as
		
		\begin{equation*}
			\mathcal{J}(u)=\int_{\mathbb{D}}|\nabla u|^2+|K(x)|\left(e^u-u\right).
		\end{equation*}

  First we claim that $\mathcal{J}$ is coercive. To prove this, consider \( u_n \in H^1(\mathbb{D}) \), with $\|u_n\|_{H^1(\D)} \to +\infty$, and we aim to prove that $\mathcal{J}(u_n) \to +\infty$. We decompose \( u_n = \bar{u}_n + \tilde{u}_n \), where
\[
\bar{u}_n = \frac{1}{|\mathbb{D}|} \int_{\mathbb{D}} u_n, \quad \text{and} \quad \int_{\mathbb{D}} \tilde{u}_n  = 0.
\]

Observe that $\|u_n\|_{H^1(\mathbb{D})}^2 =\|\nabla \tilde{u}_n\|_{L^2(\mathbb{D})}^2 + \| u_n\|_{L^2(\mathbb{D})}^2 \to +\infty$. If \( \|\nabla \tilde{u}_n\|_{L^2(\mathbb{D})} \to +\infty \), we use the inequality $e^{s}-s\geq 0$ to conclude that \[
\mathcal{J}(u_n) \to +\infty \quad \text{as} \quad  n \to +\infty.
\]
Assume now that \( \|\nabla \tilde{u}_n\|_{L^2(\mathbb{D})}\leq C\), which implies necessarily \(\|u_n\|_{L^2(\mathbb{D})} \to +\infty \). Since  \(\|u_n\|_{L^2(\mathbb{D})}\leq ||\overline{u}_n||_{L^2(\mathbb{D})} +||\tilde{u}_n||_{L^2(\mathbb{D})}\) and $||\tilde{u}_n||_{L^2(\mathbb{D})}$ is bounded by the Poincar\'{e} inequality, then \( |\bar{u}_n | \to +\infty \). Furthermore, using the inequality \( e^{s} - s \geq |s| \), we have:
\begin{equation*}
        \mathcal{J}(u_n)\geq \int_{\mathbb{D}}|K(x)|\, |u_n| = \int_{\mathbb{D}}|K(x)|\,|\bar{u}_n + \tilde{u}_n|\geq\int_{\mathbb{D}}|K(x)\bar{u}_n| -|K(x) \tilde{u}_n|.
    \end{equation*}
    Finally, by applying H\"{o}lder and Poincar\'{e} inequalities, we obtain 
    \begin{equation*}
        \mathcal{J}(u_n)\geq |\bar{u}_n|\int_{\mathbb{D}}|K(x)| -C\left(\int_{\mathbb{D}}K(x)^2\right)^{1/2}\to+\infty\,\text{ as } n\to+\infty.
    \end{equation*}
    Then, the functional  \( \mathcal{J}\) is coercive as claimed.

		Now, fixing a sufficiently small $\varepsilon_0>0$ and consider $\varepsilon\in(0, \varepsilon_0)$, we have that $\mathcal{I}_\varepsilon$ shares the same min-max geometry of $\mathcal{I}$, i.e.,
		
		\begin{equation*}
			c_{\varepsilon}:=\inf_{\gamma\in\Gamma}\max_{t\in[0,1]}\mathcal{I}_{\varepsilon}\left(\gamma(t)\right)>\max\left\{\mathcal{I}_{\varepsilon}(u_0), \mathcal{I}_{\varepsilon}(u_1)\right\},
		\end{equation*}
		where $\Gamma$ stands for the class of admissible curves,
		
		\begin{equation*}
			\Gamma=\left\{\gamma:[0,1]\to H^1(\mathbb{D}) \mbox{ continuous }: \gamma(0)=u_0, \gamma(1)=u_1\right\}.
		\end{equation*}
		
		Observe that $\mathcal{J}(u)$ is strictly positive and coercive, and that $\mathcal{I}'_\e$, $\mathcal{I}''_\e$ are uniformly H\"{o}lder continuous in bounded sets. We can now apply \cite[Proposition 4.2]{BellazziniRuiz23} or \cite[Theorem 1]{BorthwickChangJeanjeanSoave-p} to conclude the existence of a bounded Palais-Smale sequence at level $c_\e$, $\{v_k^{\e}\}$, for all $\e \in D \subset (0, \e_0)$ a set of full measure. Moreover,
		
		\begin{equation} \label{aqui} \max \{ dim\, Y: \ Y \subset H^1_G(\D) \mbox{ vector space: }\  \mathcal{I}_\e''(v_k^\e) (\phi, \phi) \leq \delta_k \| \phi\|^2 \, \, \forall \phi \in Y \} \leq 1, \end{equation}
		for some sequence $\delta_k \to 0$. 
		
		Up to a subsequence we can assume that $v_k^\e \weakto v_\e$ in $H^1_G(\D)$. By compactness we conclude that $v_\e$ is a critical point of $\mathcal{I}_\e$.  
		
		Take now $\e_n \in D$, $\e_n \to 0$, and $u_n= v_{\e_n}$. Clearly, $u_n$ is a solution of \eqref{eq: perturbed problem 1} with:

		\begin{equation}\label{aprox}
			\tilde{K}_n=\frac{|K(x)|\varepsilon_n}{1+2\varepsilon_n}; \quad K_n=\frac{|K(x)|\left(1+\varepsilon_n/2\right)}{1+2\varepsilon_n}; \quad \tilde{h}_n=\frac{1}{1+2\varepsilon_n}; \quad  h_n= \frac{h(x)}{1+2\varepsilon_n}.
		\end{equation}

		Moreover we have the Morse index bound \eqref{morse<1} as a consequence of \eqref{aqui}.
	\end{proof}
	
	\section{Blow-up analysis and proof of Theorem \ref{thm: existence}} \label{sec:Blow-upAnalysis}
	The results obtained in Section \ref{sec:Existence-result} motivate the study of blow-up solutions and the conditions for which such sequences are bounded from above. In this regard, we give the following general result, which is closely related to \cite[Theorem 1.4]{LopezSorianoMalchiodiRuiz2019}, see Remark \ref{rem: S0empty} below. 	
	
	\begin{theorem}\label{thm: Blowup}
		Let $u_n$ be a blowing-up sequence  (namely, $ \sup_\D  u_n  \to +\infty$) of solutions to 
		\begin{equation}\label{ecua-compact}
			\begin{cases}
				- \Delta u_n + 2 \tilde{K}_n(x) = 2 K_n(x) e^{u_n} & \hbox{ in } \D, \\
				\frac{\partial u_n}{\partial \nu} + 2 \tilde{h}_n(x) = 2 h_n(x) e^{u_n/2} & \hbox{ on } \partial \D, 
			\end{cases}
		\end{equation}
where $\tilde{K}_n$, $\tilde{h}_n$, $K_n$, $h_n$ are $C^1$ functions with $\tilde{K}_n \to \tilde{K}$, $K_n \to K$, $\tilde{h}_n \to \tilde{h}$ and $h_n \to h$ in the $C^1$. We assume that $K_n\leq 0$ in $\mathbb{D}$ and $K<0$ on $\partial \mathbb{D}$.  Define the singular set S as
		\begin{equation*}
			S=\left\{p\in\mathbb{D}: \exists \, x_{n}\rightarrow p\text{ such that }u_{n}(x_{n})\rightarrow+\infty\right\},
		\end{equation*}
		and 
		\begin{equation*}
			\chi_n=	\int_{\mathbb{D}}\tilde{K}_n+\int_{\partial\mathbb{D}}\tilde{h}_{n}.
		\end{equation*}
		
		Then the following assertions hold true:
		
		\begin{enumerate}
			\item $S\subset\{p\in\partial\mathbb{D}:\mathfrak{D}(p)\geq1\}$.
			\item If $\int_{\mathbb{D}}e^{u_n}$ is bounded, then
			
			\begin{equation*}
				S:=\{p_1, \dots p_j\}\subset\{p\in\partial\mathbb{D}:\mathfrak{D}(p)>1\}.
			\end{equation*}
			
			In this case $|K_n|e^{u_n} \weakto \sum_{i=1}^j \b_i \delta_{p_i}$, $h_ne^{u_n/2} \weakto \sum_{i=1}^j (\b_i+2\pi) \delta_{p_i}$ for some suitable $\b_i > 0$. 
			In particular,
			$\ \chi_n \to 2 \pi j$. 
			
			\item If $\int_{\mathbb{D}}e^{u_n}$ is unbounded, there exists a unit positive measure $\sigma$ on $\mathbb{D}$ such that:
			
			\begin{enumerate}
				\item $\displaystyle\frac{|K_{n}|e^{u_n}}{\int_{\mathbb{D}}|K_{n}|e^{u_n}}\rightharpoonup\sigma$, $\displaystyle\frac{h_{n}e^{u_{n}/2}}{\int_{\partial\mathbb{D}}h_{n}e^{u_n/2}}\rightharpoonup\sigma|_{\partial\mathbb{D}}$;
				\item supp $\sigma\subset\{p\in\partial\mathbb{D}:\mathfrak{D}\geq1$, $\mathfrak{D}_\tau(p)=0\}$.
			\end{enumerate}
			
			\item Assume now that $u_n$ are $G$-symmetric functions with $G= \langle g_k \rangle$, and $ind_G(u_n)\leq m$ for some $m$ and all $n \in \N$. If $\int_{\mathbb{D}}e^{u_n}$ is unbounded, then $S=S_0\cup S_1$, where
			
			\begin{equation*}
				S_0\subset\{p\in\partial\mathbb{D}:\mathfrak{D}(p)=1, \mathfrak{D}_\tau(p)=0\} \neq \emptyset,
			\end{equation*}
			\begin{equation*}
				S_1=\{p_1,p_2,\ldots,p_j\}\subset\{p\in\partial\mathbb{D}:\mathfrak{D}(p)>1\}, j\leq \, k \, m.
			\end{equation*}
			
		\end{enumerate}
		
	\end{theorem}
	
	\begin{remark}\label{rem: S0empty}
		Theorem \ref{thm: Blowup} is basically contained in \cite[Theorem 1.4]{LopezSorianoMalchiodiRuiz2019} excepting for three aspects. First, the hypothesis of the nonpositivity of $K_n$ inside the disk. By using the maximum principle, we exclude that the blow-up phenomenon is localized in the interior, and therefore we can focus on the boundary. Second, here we have assumed boundedness of the $G$-symmetric index, which is smaller than the index in the whole space $H^1(\D)$. This question forces us to adapt the proof of \cite{LopezSorianoMalchiodiRuiz2019} to our setting. Third, in the description of $S_0$ we include the information $S_0 \neq \emptyset$. However, this was not explicitly stated in \cite{LopezSorianoMalchiodiRuiz2019}, so we include a comment on this question at the conclusion of the proof.
	\end{remark}
	
%
	
	\begin{proof}
        Take $p \in S$ as singular point. Due to the assumptions on $K$, we can take an open regular subset $\Omega\subset \overline{\Omega} \subset \mathbb{D}$ such that $K(x)<0$ in $\overline{\D} \setminus \Omega$. By the analysis carried out in \cite[Theorem 1.4]{LopezSorianoMalchiodiRuiz2019}, we infer that $S\cap\{x\in \D: \ K(x)<0\}=\emptyset$. As a consequence, $u_n<C$ on $\partial\Omega$. Let us now consider the function $v \in L^\infty\left(\Omega\right)$ which solves the following problem:
        \begin{equation*}
        \begin{cases}
            - \Delta v = 2M & \hbox{ in } \Omega, \\
            v= C & \hbox{ on } \partial \Omega.
        \end{cases}
        \end{equation*}
        where $M>0$ is such that $|\tilde{K}_n(x)|\leq M$ for all $n \in \N$. Since $K_n \leq 0$ the strong comparison principle implies that $u_n < v$ on $\overline{\Omega}$. In particular, we obtain that $u_n$ is bounded from above in $\Omega$, so that $S$ is necessarily contained in $\partial \mathbb{D}$ .

        Now, we will show that under a certain rescaling we obtain a limit solution defined in a half-plane. This is rather usual argument in the blow-up analysis of solutions of Liouville type equations, and comes back to \cite{BrezisMerle91, LiShafrir94} for the case of entire solutions. The main novelty here is that we do not know, a priori, that the quantities
		
		$$ \int_{\D} e^{u_n}, \ \int_{\partial \D} e^{u_n/2},$$ are uniformly bounded. This difficulty was addressed in \cite{LopezSorianoMalchiodiRuiz2019} in a way that will be recalled below.

		Let $p\in S$ a singular point of $u_n$ and consider a neighborhood $\mathcal{N}$ of $p$. Without loss of generality assume that $p=(1,0)$. By conformal invariance we can pass from $\mathcal{N}$ to a half ball $B_{0}^+(r)=\{(x,y)\in\mathbb{R}^2:x^2+y^2=r, y\geq0\}$. 
		This is justified by a classical argument, see for instance \cite{GidasSpruck1981}. It is possible to straighten $\Gamma:=\mathcal{N}\cap\partial\mathbb{D}$ by passing to $B^+_{0}(r)$ using a suitable diffeomorphism, and due to the conformal invariance of problem \eqref{ecua-compact} one can take this diffeomorphism as the following M\"{o}bius transform (in complex notation)
		
		\begin{equation*}\label{eq: diffeo}
			\Upsilon:\mathbb{R}^2_{+}\to\mathbb{D};\quad x\mapsto\Upsilon(x)=\frac{i-x}{x+i}.
		\end{equation*}
	Here and in what follows we denote $\R^2_+=\{(x, y) \in \R^2: \ y>0\}$. Notice that the point $p$ corresponds to the origin. Then, if we consider the function
		
		\begin{equation*}\label{eq:vnconforme}
			w_n(x):=u_n\left(\Upsilon(x)\right)+2\log|\Upsilon'(x)|=u_n\left(\Upsilon(x)\right)+2\log\frac{2}{|x+i|^2},
		\end{equation*}	
		thanks to the conformal invariance,  we have:
		
		\begin{equation*}\label{eq: probb+}
			\left\{
			\begin{array}{ll}
				-\Delta w_n+2\tilde{K}_n\left(\Upsilon(x)\right)=2K_n\left(\Upsilon(x)\right)e^{w_n} &\text{in }B_0^+(r),\\
				\\
				\frac{\partial w_n}{\partial \nu}+2\tilde{h}_n\left(\Upsilon(x)\right)=2h_n\left(\Upsilon(x)\right)e^{w_n/2}&\text{on } \Gamma_0^+(r),\\
			\end{array}\right.
		\end{equation*}
		for some small $r>0$. Here $\Gamma_0^+(r)=(-r,r)\times\{0\}$ is the straight portion of $\partial B_0^+(r)$.
		
		Now, take a sequence $y_{n}$ in $B_{0}^{+}(r)$ such that $y_{n}\rightarrow 0$ and $w_{n}(y_n)\rightarrow+\infty$. Observe that, a priori, we do not know if $p$ is isolated in $S$, and then we cannot assure that $y_n$ is a sequence of local maxima. We bypass this difficulty by using Ekeland's variational principle, which provides us with a sequence $x_n\in B_{0}^{+}(r)$ such that
		\begin{enumerate}
			\item $w_n(x_n)\leq w_n(y_n)$,
			\item $d(x_{n},y_n)\leq\sqrt{\varepsilon_{n}}$,
			\item $e^{\frac{-w_{n}(x_{n})}{2}}\leq e^{\frac{-w_{n}(z)}{2}}+\sqrt{\varepsilon_{n}}d(x_{n},z)$ for every $z\neq x_{n}$,
		\end{enumerate}
		where $\varepsilon_{n}=e^{\frac{-w_{n}(y_{n})}{2}}$.
		
		As a consequence $x_{n}\rightarrow 0$ and $w_{n}(x_n)\rightarrow+\infty$. Thanks to the point ($3$), this new sequence is convenient to rescale and move to a limit problem. Indeed, it was proved in \cite[Proposition 5.1]{LopezSorianoMalchiodiRuiz2019} that, up to subsequence,
		
		\begin{equation}\label{eq:convergencevn}
			v_n\rightarrow v\text{ in }\mathcal{C}^{2}_{loc}\left(\mathbb{R}\times(-t_0,+\infty)\right),
		\end{equation}
		where $v_n$ is a suitable rescaled function defined as
		
		\begin{equation*}\label{eq:vn}
			v_n(x):=w_n(\delta_n x+x_n)+2\log\delta_n,
		\end{equation*}	
		for $x\in B_n:=B_{0}^+(\frac{r}{\delta_n})\cap B_{-\frac{x_n}{2\delta_n}}(\frac{r}{\delta_n})$ and $\delta_{n}=e^{\frac{-w_n(x_n)}{2}}\rightarrow 0$.  
		In what follows we distinguish two cases:
		
		\underline{Case 1}:
		\begin{equation*}
			d(x_n,\Gamma_0^{+}(r))=O(\delta_n)\quad\text{ as }n\to+\infty.
		\end{equation*}
		Passing to a subsequence we can assume that $\frac{d(x_n,\Gamma_0^{+}(r))}{\delta_n}\to t_0\geq0$. Then, the function $v_n$ solves
		
		\begin{equation}\label{eq: rescaledproblem}
			\left\{
			\begin{array}{ll}
				-\Delta v_n+2\delta_n^2\tilde{K}_n\left(\delta_n\Upsilon(x)+x_n\right)=2K_n\left(\delta_n\Upsilon(x)+x_n\right)e^{v_n} &\text{in }B_n,\\
				\\
				\frac{\partial v_n}{\partial \nu}+2\delta_n\tilde{h}_n\left(\delta_n\Upsilon(x)+x_n\right)=2h_n\left(\delta_n\Upsilon(x)+x_n\right)e^{v_n/2} &\text{on }\Gamma_n,\\
			\end{array}\right.
		\end{equation}
		where $\Gamma_n$ is the straight portion of $\partial B_n$. In addition, thanks to \eqref{eq:convergencevn}, up to a translation, $v$ is a solution of the limit problem
		
		\begin{equation}
			\left\{
			\begin{array}{ll}
				-\Delta v=2K(0)e^{v}&\text{in }\mathbb{R}^{2}_{+},\\
				\\
				\frac{\partial v}{\partial \nu}=2h(0)e^{v/2}&\text{on }\partial\mathbb{R}^{2}_{+}.
			\end{array}\right.
			\label{eq: limit problem}
		\end{equation}

%
		Let us recall that last problem admits solutions only if $\mathfrak{D}_{0}\geq1$, and this proves $(1)$. In particular, if $\mathfrak{D}_0=1$ the only solutions of \eqref{eq: limit problem} are given by: 
		
		\begin{equation*}
			v_\lambda(s,t)=2\log\left(\frac{\lambda}{1-\lambda}\right)-\log|K_0|,\quad \lambda>0.    
		\end{equation*}
		For the case $\mathfrak{D}_0>1$ there are many different solutions, which have been classified in \cite{GalvezMira2009}.
		
		\underline{Case 2}: 
		\begin{equation*}
			\frac{d(x_n,\Gamma_0^{+}(r))}{\delta_n}\to+\infty,\quad\text{ as }n\to+\infty.
		\end{equation*}
		
		In this situation the rescaled domains $B_n$ invade all $\R^2$. Hence, reasoning as before, up to subsequence we have 
		
		$$v_n\rightarrow v\text{ in }\mathcal{C}^{2}_{loc}\left(\mathbb{R}^2\right),$$
		which is a solution of the equation
		
		$$-\Delta v=2K(0)e^{v} \quad\text{in }\mathbb{R}^{2}.$$
		If $K(0)<0$ it is well know, via Liouville's formula, that the above problem does not admit any solution.\\
		
		We now prove $(4)$, assume that $u_n$ are $G$-symmetric functions with $G= \langle g_k \rangle$, and $ind_G(u_n)\leq m$ for some $m$ and all $n \in \N$. Next lemma uses the Morse index bound on $u_n$ to reduce the possible solutions $v$. In order to do that, we define the Morse index of $v$ as:
		
		\begin{equation*}
			ind(v)= \sup \{dim\, E, \ E \subset C_0^{\infty}(\overline{\R^2_+}): \ \tilde{Q}|_E \mbox{ is negative definite}\},
		\end{equation*}
		where $\tilde{Q}$ is defined as:
		
		\begin{equation*}\label{eq: QuadraticLimit}
			\tilde{Q}(\psi)=\int_{\mathbb{R}^{2}_{+}}\left(|\nabla\psi|^2+2\psi^2K(0)e^{v}\right)-\int_{\partial\mathbb{R}^{2}_{+}}\psi^2h(0) e^{v/2}.
		\end{equation*}
		
		Here $C_0^{\infty}(\overline{\R^2_+})$ stands the set of test functions whose support is bounded.

		\begin{lemma}\label{lem: Blowup}
			If $p \in S_1$, the Morse index of $v$ is $1$. Moreover, $S_1=\{p_1, \dots, p_j\}$ with $j \leq k \, m$.
		\end{lemma}
		
		\begin{proof} 
			By contradiction, suppose that $ind(v)>1$. By \cite[Theorem 4.2]{LopezSorianoMalchiodiRuiz2019}, $ind(v)=+\infty$; this means that for any $l \in \N$, there exists a vector space of dimension $l$, namely $E\subset\mathcal{C}_{0}^{\infty}\left(\mathbb{R}^2_+\right)$, such that 
			
			\begin{equation*}
				\tilde{Q}(\psi)=\int_{\mathbb{R}^{2}_{+}}\left(|\nabla\psi|^2+2\psi^2K(0)e^{v}\right)-\int_{\partial\mathbb{R}^{2}_{+}}\psi^2h(0) e^{v/2}<-\varepsilon^2,
			\end{equation*}
			for all \( \psi \in E \) with \( \|\psi\|_{H^1(\mathbb{D})} = 1 \).
            
			On the other hand, let us define the subset
			
			\begin{equation*}
				\mathcal{M}:=\left\{z\in\mathbb{D}: \Upsilon^{-1}\left(z\right)\delta_n+x_n\in\supp(\psi)\right\}\subset\bar{\mathbb{D}},
			\end{equation*}
			which is, for $n$ large enough, a neighborhood of $p=(1,0)$.  Moreover, let $G= \langle g_k \rangle$ be a symmetric group of order $k$ such that every $g^{l}(\mathcal{M})$ and $g^{j}(\mathcal{M})$ are disjoint subsets, $0\leq l<j\leq k-1$. Now consider the following function
			
			\begin{equation*}
				\psi_{n}:=\psi\left(\Upsilon^{-1}(z)\delta_n+x_n\right) \quad z\in\mathbb{D}.
			\end{equation*}
			Recalling the expression \eqref{eq:quadraticform}, following the argument of straightening, the scaling variables and the uniform convergence on compacts, we have
			
			\begin{equation*}
				\begin{split}
					Q_n(\psi_n)=&\int_{\mathbb{D}}\left(|\nabla\psi_n|^{2}-2\psi_n^{2}K_n(z)e^{u_n}\right)-\int_{\partial\mathbb{D}}\psi_n^{2}h_n(z)e^{u_n/2}\\
					=&\int_{\mathcal{M}}\left(|\nabla\psi_n|^{2}-2\psi_n^{2}K_n(z)e^{u_n}\right)-\int_{\partial\mathcal{M}}\psi_n^{2}h_n(z)e^{u_n/2}\\
					=&\int_{\Upsilon^{-1}(\mathcal{M})}\left(|\nabla\left(\psi_n\circ\Upsilon\right)|^{2}-2(\psi_n\circ\Upsilon)^{2}\left(K_n\circ\Upsilon\right)(z)e^{w_n}\right)\\
					&-\int_{\Upsilon^{-1}(\partial\mathcal{M})}(\psi_n\circ\Upsilon)^{2}\left(h_n\circ\Upsilon\right)(z)e^{w_n/2}\\
					=&\int_{B_n}\left(|\nabla\psi_n\left(\Upsilon\left(z\delta_n+x_n\right)\right)|^{2}-2\left(\psi_n\left(\Upsilon\left(z\delta_n+x_n\right)\right)\right)^{2}K_n\left(\Upsilon\left(z\delta_n+x_n\right)\right)e^{v_n}\right)\\
					&-\int_{\Gamma_n}\left(\psi_n\left(\Upsilon\left(z\delta_n+x_n\right)\right)\right)^{2}h_n\left(\Upsilon\left(z\delta_n+x_n\right)\right)e^{v_n/2}\\
					=&\int_{\mathbb{R}^{2}_{+}}\left(|\nabla\psi|^2+2\psi^2K(0)e^{v}\right)-\int_{\partial\mathbb{R}^{2}_{+}}\psi^2h(0) e^{v/2}<-\varepsilon^2.
				\end{split}
			\end{equation*}	
			where $\partial\mathcal{M}:=\mathcal{M}\cap\mathbb{D}$. Then one can construct the following $G$-symmetric function
			\begin{equation*}
				\Psi_{n}:=\sum_{j=0}^{k-1}\psi_{n}\left(g^{j}(z)\right),
			\end{equation*} 
			which is well defined since by definition $\psi_n$ has compact and disjoint supports. Therefore, taking advantage of the symmetric group $G$ and following the straightened and scaled argument, it allows us to see that
			\begin{equation*}
				\begin{split}
					Q_{n}(\Psi_{n})&=\int_{\mathbb{D}}\left(|\nabla\Psi_n|^{2}-2\Psi_n^{2}K_n(z)e^{u_n}\right)-\int_{\partial\mathbb{D}}\Psi_n^{2}h_n(z)e^{u_n/2}\\
					&=\sum_{j=0}^{k-1}\left(\int_{g^{j}(\mathcal{S})}\left(|\nabla\Psi_n|^{2}-2\Psi_n^{2}K_n(z)e^{u_n}\right)-\int_{g^{j}(\mathcal{S}\cap\partial\mathbb{D})}\Psi_n^{2}h_n(z)e^{u_n/2}\right)\\
					&=k\left(\int_{\mathcal{M}}\left(|\nabla\psi_n|^{2}-2\psi_n^{2}K_n(z)e^{u_n}\right)-\int_{\partial\mathcal{M}}\psi_n^{2}h_n(z)e^{u_n/2}\right)\\
					&=kQ_n(\psi_n)<-\varepsilon^2.
				\end{split}
			\end{equation*}
		 Moreover, $span \{ \Psi_n: \ \psi \in E \}$ is clearly a vector space of dimension $l$. By taking $l >m$, we obtain a contradiction with $ind_G(u_n)\leq m$.

The aim now is to prove that $|S_1|\leq k m$. Consider $u_n\in H^{1}_{G}(\mathbb{D})$ a sequence of solutions to the problem \eqref{ecua-compact} with $ind_G(u_n)=m$. Let us argue again by contradiction and suppose that $|S_1|>k\cdot m$. Due to the $G$-symmetry, there exists at least $(m+1)k$ singular points of $u_n$ in $S_1$, namely $p_i$ with $i=1,\ldots, (m+1)k$. 
			
			Moreover, if we divide $\mathbb{D}$ into $k$ identical disk sectors $T_s$, where $s=1,\ldots, k$, according to the order of the group $G$ in each sector there must be at least $m+1$ singular points. Therefore, without loss of generality, we can consider $m+1$ singular points $p_j\in S\cap T_1$ of $u_n$ in $T_1$. We can assume that
			
			\begin{equation*}
				B_{p_j}^{+}(r_j)\cap B_{p_l}^{+}(r_l)=\emptyset,\quad\text{for all } 1\leq j<l\leq m+1.
			\end{equation*}
			
			As before, one can rescale the solutions $u_n$ in each neighborhood until one obtains a profile $v$ which is a solution of the boundary problem \eqref{eq: limit problem} and whose Morse is equal to $1$. But if so, one can take $m+1$ compactly supported functions $\psi_j$ defined in respective $1$-dimensional vector spaces of $\mathcal{C}_0^\infty(\mathbb{R}^2_+)$, such that
			
			\begin{equation*}
				\tilde{Q}\left(\psi_j\right)<-\varepsilon^2_j,\quad\text{for all } 1\leq j\leq m+1.
			\end{equation*}
			We are able to define a $m+1$-dimensional vector space $\tilde{E}:=span(\psi_1,\psi_2,\ldots, \psi_{m+1})$ generated by $m+1$ compactly-supported functions such that the above quadratic form would be negative-definite on $\tilde{E}$.  Following the same arguments as in the proof of the first assertion one can conclude that $ind_G(u_n)\geq m+1$, reaching a contradiction.
		\end{proof}
		Finally we show that under the assumptions of Theorem \ref{thm: Blowup} (4),  $S_0 \neq \emptyset$. To see this, we recall the following result proved in \cite[Lemma 7.4]{LopezSorianoMalchiodiRuiz2019}
		\begin{lemma}
			Let $p\in S_1$, then there exist fixed constants $r,C>0$ such that
			\begin{equation*}
				\int_{B_{p}(r)\cap\mathbb{D}}e^{u_n}\leq C;\quad\int_{B_{p}(r)\cap\partial\mathbb{D}}e^{u_n/2}\leq C;\quad u_{n}\rightarrow-\infty\text{ on }\partial B_{p}(r)\cap\mathbb{D}.
			\end{equation*}
		\end{lemma}
		
		\medskip 
		
		Suppose $\int_{\mathbb{D}}e^{u_n}$ is unbounded and $S_0=\emptyset$, due to last lemma for all $p_i\in S$ we have that 
		
		\begin{equation*}
			\int_{\mathbb{D}}e^{u_n}=\sum_{i=1}^{km}\left(\int_{B_{p_{i}}(r_i)\cap\mathbb{D}}e^{u_n}+\int_{\mathbb{D}\setminus\left(\cup_{i=1}^{km}B_{p_i}(r_i)\right)}e^{u_n}\right)\leq C,
		\end{equation*}
		for suitable $r_i>0$ and  $i=1,\ldots,km$. This contradicts the unbounded hypothesis then $S_{0}\neq\emptyset$.
	\end{proof}
	
	\begin{proof}[\underline{Proof of Theorem \ref{thm: existence}}] We will divide the proof in two parts, according to the types of symmetry groups we have been working with.
		
		\medskip {\bf 	Case 1: $G = \langle g_k \rangle$}  
		
		We consider the sequence $u_n$ given by Proposition \ref{prop:ConvergenceSubsequence}, and we claim that they are uniformly bounded from above. From this one readily obtain compactness of solutions, and hence we can get a solution of \eqref{eq: Problem 1} in the limit.

Reasoning by contradiction, assume that $\max \{u_n\} \to + \infty$, and apply Theorem \ref{thm: Blowup}. Since the set:
		\begin{equation*}
			\{p\in\partial\mathbb{D}:\mathfrak{D}(p)=1, \mathfrak{D}_\tau(p)=0\}
		\end{equation*}
		is empty, we conclude that $\int_{\D} e^{u_n}$ is bounded. By symmetry, we conclude that $S_1$ contains at least $k$ points. However, with the quantization given by
the second point of Theorem~\ref{thm: Blowup}, one obtains
		\begin{equation*}
			\chi_n=	\int_{\mathbb{D}}\tilde{K}_n+\int_{\partial\mathbb{D}}\tilde{h}_{n} \to 2 k\pi,
		\end{equation*}
		which is a contradiction.
		
		\medskip {\bf 	Case 2: $G=O(2)$}
		
		In this case problem \eqref{ecua-compact} can be written as:
		
		\begin{equation*}
			\left\{
			\begin{array}{ll}
				-u''_n-\frac{u'_n}{r}+2\tilde{K}_n=2K_n(r)e^{u_n}&\text{in }0<r<1,\\
				\\
				u'_n(1)+2\tilde{h}_n=2h_n(1)e^{u_n(1)/2},&\\
				\\
				u'_n(0)=0.
			\end{array}\right.
			\label{eq: perturbed radial problem 1}
		\end{equation*}
		
		
		Multiplying the main equation by $u'_n$ and integrating you get:
		\begin{equation*}
			-\frac{u'_n(1)^2}{2}-\int_{0}^{1}\frac{(u'_n)^2}{r}+2\tilde{K}_n\left(u_n(1)-u_n(0)\right)=\int_{0}^{1}2K_n(r)u'_ne^{u_n}.
		\end{equation*}
		
		Now, integrating by parts and considering the boundary conditions we can write each of the sides of the above equation as
		\begin{equation*}
			\begin{split}
				A_1:=&-2h_n^2(1)e^{u_n(1)}-2\tilde{h}_n^2+4h_n(1)\tilde{h}_ne^{u_n(1)/2}-\int_{0}^{1}\frac{(u'_n)^2}{r}+2\tilde{K}_n\left(u_n(1)-u_n(0)\right),\\
				A_2:=&2K_n(1)e^{u_n(1)}-2K_n(0)e^{u_n(0)}-\int_{0}^{1}2K'_n(r)e^{u_n}.
			\end{split}
		\end{equation*}
		
		After grouping the terms, $A_2-A_1=0$, we have the following inequality
		\begin{equation*} 
			2e^{u_n(1)}\left(K_{n}(1)+h_n^2(1)\right)-4h_n(1)\tilde{h}_ne^{u_n(1)/2}-2\tilde{K}_n\left(u_n(1)-u_n(0)\right)\leq \int_{0}^{1}2K'_n(r)e^{u_n}.
		\end{equation*}
		
		Clearly,
		
		$$ 	\int_{0}^{1}2K'_n(r)e^{u_n} \leq 2 \max{ \frac{K_n'}{|K_n|}  }\int_0^1 |K_n| e^{u_n} = 2 \max{ \frac{K_n'}{|K_n|}  } (h_n(1) e^{u_n(1)/2} - \chi_n),$$
		where $\chi_n \to 2 \pi$. Then,
		
		\begin{align*}
			2e^{u_n(1)}\left(K_{n}(1)+h_n^2(1)\right)   \\ \leq 4h_n(1)\tilde{h}_ne^{u_n(1)/2}+2\tilde{K}_n\left(u_n(1)-u_n(0)\right)+  2 \max{ \frac{K_n'}{|K_n|}  } (h_n(1) e^{u_n(1)/2} - \chi_n).
		\end{align*}
		Now recall that $u_n(0)$ is bounded, $\tilde{h}_n \to 1$, $\tilde{K}_n \to 0$ and $\mathfrak{D}>1$, and then we obtain that $e^{u_n(1)}$ is uniformly bounded, and then $u_n(1) \leq C$ for some $C \in \R$. Then, the maximum principle implies that $u_n \leq \bar{u}$, where $\bar{u}$ solves:

		\begin{equation*}
			\left\{
			\begin{array}{ll}
				- \Delta \bar{u} = \sup_{n} \| \tilde{K}_n \|_{L^{\infty}} &\text{ in } \D,\\
				\\
				\bar{u}= C & \text{ in } \partial \D.
			\end{array}\right.
		\end{equation*}
		
		This implies that $u_n$ is bounded from above and we conclude.
		
	\end{proof}
	
	
	\section{Proof of Theorem \ref{thm: nonexistence}}
In this section we proof the non existence result of Theorem \ref{thm: nonexistence}. Observe that if $u$ is a solution of \eqref{eq: Problem 1}, then we can make the change of scale $v= u+c$ to obtain:
	
	\begin{equation*}
		\left\{
		\begin{array}{ll}
			-\Delta v=2 \lambda^2 K(x)e^{v}&\text{in }\mathbb{D},\\
			\\
			\frac{\partial v}{\partial \nu}+2=2 \lambda h(x)e^{v/2}&\text{on }\partial\mathbb{D},\\
		\end{array}\right.
	\end{equation*}
	with $\lambda= e^{-c/2}$. Then, we can reduce ourselves to the case $c_0=1$ in Theorem \ref{thm: nonexistence}. 
	
	The next lemma will be of use:
	\begin{lemma}\label{lem: sub-sup}
		If the problem (\ref{eq: Problem 1}) is solvable with $C^{0,\alpha}$ functions $K$, $h$, and  $K(x) \leq -1$, then there exists $\tilde{u} \in C^{\infty}(\D) \cap C^{1, \alpha} (\overline{\D})$ a solution of:
		
		\begin{equation}
			\left\{
			\begin{array}{ll}
				-\Delta \tilde{u}= - 2 e^{\tilde{u}}&\text{in }\mathbb{D},\\
				\\
				\frac{\partial \tilde{u}}{\partial \nu}+2=2 \tilde{h}(x)e^{v/2}&\text{on }\partial\mathbb{D},\\
			\end{array}\right.
		\end{equation} \label{K=-1}
		where $\tilde{h}(x) \leq h(x)$ for all $x \in \partial \D$. 
	\end{lemma}
	
	\begin{proof} Observe that a solution $u$ of \eqref{eq: Problem 1} solves:
		
		\begin{equation*}
			-\Delta {u}(x)\leq-2e^{{u}} \qquad \text{ in }\mathbb{D}.
		\end{equation*}
		Now, let $\bar{u}\in\mathcal{C}^2(\mathbb{D})$ such that
		
		\begin{equation*}
			\left\{
			\begin{array}{ll}
				-\Delta\bar{u}(x)=0 &\text{in }\mathbb{D},\\
				\\
				\bar{u}= u &\text{on }\partial\mathbb{D}.\\
			\end{array}\right.
		\end{equation*}
		
		By the maximum principle,  $\bar{u}\geq u $. It follows from sub-super solution method, (see for instance \cite[Chapter 9]{Evans2010}) that we may find a solution $\tilde{u}$ of the problem:
		
		\begin{equation*}
			\left\{
			\begin{array}{ll}
				-\Delta \tilde{u}= - 2 e^{\tilde{u}}&\text{in }\mathbb{D},\\
				\\
				\tilde{u} = u &\text{on }\partial\mathbb{D}.\\
			\end{array}\right.
		\end{equation*}
		
		Clearly, $\tilde{u}$ solves \eqref{K=-1} with:
		
		$$ \tilde{h}= e^{-\tilde{u}/2} \left(\frac{\partial \tilde{u}}{\partial\nu} + 2\right).$$
		
		Moreover, $u \leq \tilde{u} \leq \bar{u}$ on $\mathbb{D}$, which implies that:
		
		$$\frac{\partial \tilde{u}}{\partial\nu}(x)\leq\frac{\partial u }{\partial \nu}(x) = 2 h(x) e^{u/2} - 2= 2 h(x) e^{\tilde{u}/2} - 2 \quad \text{ on } \partial \D.$$
		This implies that $\tilde{h} \leq h$. The regularity of $\tilde{u}$ follows from classical Schauder estimates. 
	\end{proof}

	\begin{proof}[\underline{Proof of Theorem \ref{thm: nonexistence}}]
		By contradiction, suppose there exists a solution for the semilinear elliptic problem (\ref{eq: Problem 1}) with $K(x)\leq-1$ and $h(x)\leq 1$. By Lemma \ref{lem: sub-sup}, consider a solution $\tilde u$ to problem \eqref{K=-1} with $\tilde{h}(x) \leq -1$ for all $x \in \partial \D$. By a classical result mainly due by Liouville \cite{Liouville1853} (see also \cite[Section 3.2]{Yang2011}), there exists a locally univalent holomorphic function $g$ (with $|g|<1$) in $\mathbb{D}$ such that 
		
		\begin{equation}
			\tilde u=\log{\frac{4|g'|^2}{(1-|g|^2)^2}}.
			\label{eq: Liouville Sol a=-2}
		\end{equation}
		
		Now note that, by (\ref{eq: Liouville Sol a=-2}), $e^{\tilde u}=\frac{4|g'|^2}{(1-|g|^2)^2}$. It is important to point out that since $\tilde u \in C^{1}(\overline{\D})$, then $g \in C^{2}(\overline{\D})$ (see the derivation of $g$ given in \cite[Section 3.2]{Yang2011}). Moreover,
		
		\begin{equation*}
			g:\left(\bar{\mathbb{D}},m_{0}\right)\longrightarrow\left(\mathbb{P},m_{\mathbb{P}}\right),
		\end{equation*}
		is a local isometry. Here $m_{0}=e^{\tilde u}$, $\left(\mathbb{P},m_{\mathbb{P}}\right)$ denotes the Poincar\'{e} disk model for 2-dimensional hyperbolic metric spaces, so that $\mathbb{H}$ is the unit disk and $m_{\mathbb{P}}=\frac{4|dz|^2}{(1-|z|^2)^2}$ is the Poincar\'{e} metric defined on it. 
		
		%
		%
		
		Then we can consider $\Omega:=g(\bar{\mathbb{D}})\subset\mathbb{P}$. Now, we denote $D_r$ a disk centered at the origin with euclidean radius $r \in (0,1)$, and $C_r$ its boundary. If $r$ is sufficiently close to $1$, we have that $\Omega\subset D_r$. We make $r$ decrease up to a first contact, that is, we take $r$ such that:
		
		$$ C_r \cap g(\overline{\D}) \neq \emptyset, \ g(\overline{\D}) \subset D_r.$$
		
		Take $p \in C_r \cap g(\overline{\D})$. Clearly $p \in \partial g(\overline{\D})$: since $g$ is open, there exists $q \in \partial \D$ with $g(q)=p$. Moreover $ \gamma = g(\partial \D)$ is an inmersed $C^2$ curve inside $\overline{D}_r$, tangential to $C_r$ at $g(q)$. Since $g$ is a local isometry, $h(g(q)) \leq 1$, but this is a contradiction since the geodesic curvature of $C_r$ is $\frac{r^2+1}{2r}>1$ (see Figure \ref{fig:Poincare_Model}). 
		
		\begin{figure}[ht]
			\centering
			\includegraphics[width=9.5cm]{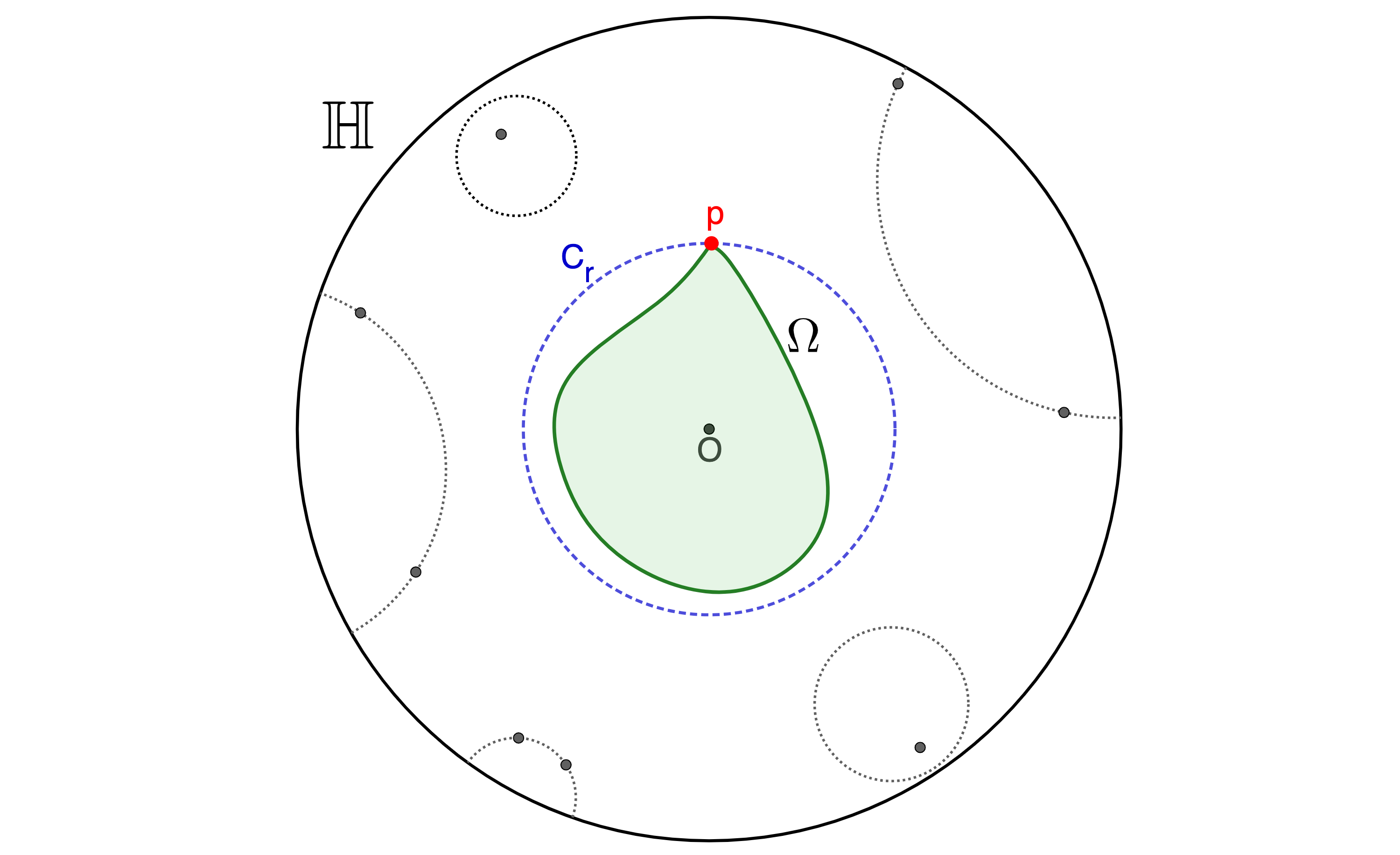}
			\caption{Construction in $\mathbb{H}$.}
			\label{fig:Poincare_Model}
		\end{figure}

	\end{proof}
	
	\appendix
	\numberwithin{equation}{section}
	\section{Tests functions}\label{sec: appendix}
	In this section, we detail the proof of two lemmas, which together imply the validity of Lemma \ref{lem: bubble}. We divide the discussion into two parts, according to the types of symmetry groups $G$ that we have chosen throughout this paper.\\
	
	\textbf{Case 1:} $\mathbf{G = \langle g_k \rangle.}$ Consider $k$-points $p_i$ located on $\partial\mathbb{D}$ and let $q_i=p_i+rn(p_i)$, where $n$ is the outward normal vector to $\partial\mathbb{D}$ and $r>0$ is small enough such that $q_i$ belongs to a regular extension of $\mathbb{D}$. Given a parameter $\mu$ such that $\mu d(x,q_i)>1$ for every $x\in\mathbb{D}$, we define the functions 
	
	\begin{equation*}
		\label{eq: phis}
		\begin{array}{cc}
			\varphi_{\mu,q}:\mathbb{D}\rightarrow\mathbb{R}, & \varphi_{\mu,q}(x)=\log\frac{4\mu^2}{\left(\mu^2d^2(x,q)-1\right)^2},\\ \\

			\tilde{\varphi}_{\mu,q}:\mathbb{D}\rightarrow\mathbb{R}, & \tilde{\varphi}_{\mu,q}(x)=\varphi_{\mu,q}(x)-\log|K(x)|,
		\end{array}
	\end{equation*}
	and consider also the functions 
	
	\begin{equation}\label{eq: LogSumgofBubbles}
		\begin{array}{cc}
			\varPhi_{\mu,k}:\mathbb{D}\rightarrow\mathbb{R}, & \varPhi_{\mu,k}(x)=\log{\left(\sum_{i=1}^{k}e^{\varphi_{\mu,q_i}(x)}\right)}, \\ \\
			
			\tilde{\varPhi}_{\mu,k}:\mathbb{D}\rightarrow\mathbb{R}, & \tilde{\varPhi}_{\mu,k}(x)={\varPhi}_{\mu,k}(x)-\log|K(x)|.
		\end{array}
	\end{equation}
	
	\begin{lemma}\label{lem:testrota}
		Let $p\in\partial\mathbb{D}$ such that $\mathfrak{D}(p)>1$ and let $\mathcal{I}$ be as in \eqref{eq: Energy Functional Disk}. Then there exists $r>0$ and a function $\tilde{\varPhi}_{\mu,k}$ be defined in \eqref{eq: LogSumgofBubbles} such that
		\begin{equation}
			\label{eq: bubble}
			\mathcal{I}\left(\tilde{\varPhi}_{\mu,k}\right)\rightarrow-\infty,\quad\int_{\partial\mathbb{D}}e^{\tilde{\varPhi}_{\mu,k}/2}\rightarrow+\infty\quad as\quad\mu\rightarrow\frac{1}{d(p,q)}=\frac{1}{r}.
		\end{equation}
		
		\begin{proof} 
			
			For convenience, we will rewrite the energy functional as follows
			
			\begin{equation*}
				\begin{split}
					\mathcal{I}\left(\tilde{\varPhi}_{\mu,k}\right)&=\int_{\mathbb{D}}\left(\frac{1}{2}|\nabla \tilde{\varPhi}_{\mu,k}|^2-2Ke^{\tilde{\varPhi}_{\mu,k}}\right)+\int_{\partial\mathbb{D}}\left(2\tilde{\varPhi}_{\mu,k}-4h\sqrt{e^{\tilde{\varPhi}_{\mu,k}}}\right)\\
					&=\int_{\mathbb{D}}\left(\frac{1}{2}|\nabla \tilde{\varPhi}_{\mu,k}|^2+2e^{{\varPhi}_{\mu,k}}\right)+\int_{\partial\mathbb{D}}\left(2\tilde{\varPhi}_{\mu,k}-4\mathfrak{D}\sqrt{e^{{\varPhi}_{\mu,k}}}\right).
				\end{split}
			\end{equation*}
			
			Note that, letting $\varepsilon>0$ and using the Young's inequality we can estimate the first term of last equation as
			
			\begin{equation}\label{eq: gradientestimate}
				\frac{1}{2}\int_{\mathbb{D}}|\nabla\tilde{\varPhi}_{\mu,k}|^2\leq\left(\frac{1}{2}+\varepsilon\right)\int_{\mathbb{D}}|\nabla \varPhi_{\mu,k}|^2+O(1).
			\end{equation}
			
			We claim that the function $\varPhi_{\mu,k}$ defined in (\ref{eq: LogSumgofBubbles}) satisfies the following estimates:
			
			\begin{equation}\label{eq: GradientofLog}
				\int_{\mathbb{D}}|\nabla\varPhi_{\mu,k}|^2\leq\frac{8k\pi}{\sqrt{\mu^2 r^2-1}}+o\left(\frac{1}{\sqrt{\mu^2 r^2-1}}\right),
			\end{equation}
			
			\begin{equation}\label{eq: ExpofLog}
				\int_{\mathbb{D}}e^{\varPhi_{\mu,k}}\leq\frac{2k\pi\mu r}{\sqrt{\mu^2 r^2-1}}+o\left(\frac{1}{\sqrt{\mu^2 r^2-1}}\right),
			\end{equation}
			
			\begin{equation}\label{eq: ExpofLogBoundary}
				\int_{\partial\mathbb{D}}\mathfrak{D}e^{\frac{\varPhi_{\mu,k}}{2}}\geq \min_{B_p(r)\cap\partial\mathbb{D}}\mathfrak{D}\frac{2k\pi}{\sqrt{\mu^2r^2-1}}+o\left(\frac{1}{\sqrt{\mu^2r^2-1}}\right),
			\end{equation}
			
			\begin{equation}\label{eq: LogofPhis}
				\int_{\partial\mathbb{D}}\varPhi_{\mu,k}\leq O(1).
			\end{equation}
			
			From these, the assertion of (\ref{eq: bubble}) follows immediately.

			\medskip It should be noted that in the following we obtain integral expressions that have already been computed in the appendix of \cite{LopezSorianoMalchiodiRuiz2019}. However, for the sake of completeness, we include below the whole argument.\\
			
			\underline{Estimate (\ref{eq: GradientofLog})}.
			First of all, note that the gradient of $\varPhi_{\mu,k}$ is given by
			
			\begin{equation}\label{eq: GradientPhi}
				\nabla\varPhi_{\mu,k}=\frac{\sum_{i=1}^{k}\nabla\varphi_{\mu,q_i}e^{\varphi_{\mu,q_i}}}{\sum_{i=1}^{k}e^{\varphi_{\mu,q_i}}}.
			\end{equation}
			
			We can split the integral (\ref{eq: GradientofLog}) as
			
			\begin{equation}
				\label{eq: SplitGradientofSum}
				\int_{\mathbb{D}}|\nabla\varPhi_{\mu,k}|^2=\sum_{i=1}^k\int_{A_i}|\nabla\varPhi_{\mu,k}|^2+\int_{\mathbb{D}\setminus(\cup_{i=1}^{k}A_i)}|\nabla\varPhi_{\mu,k}|^2,
			\end{equation}
			where $A_i=\mathbb{D}\cap B_{q_i}(2r)$, and $B_{q_i}(2r)$ denote a ball of center $q_i$ and radius $2r$. 
			
			\begin{figure}[ht]
				\centering
				\includegraphics[width=9.5cm]{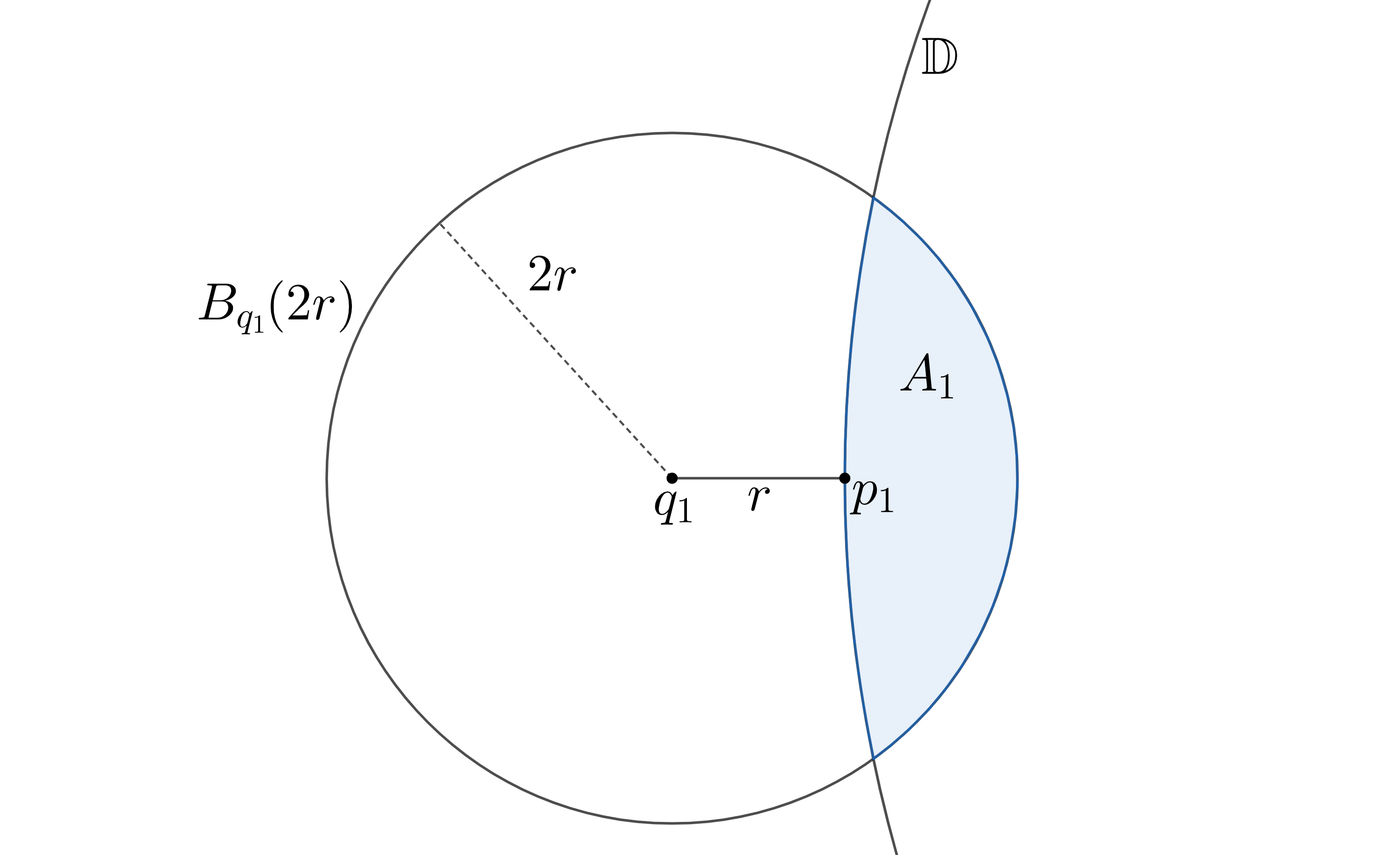}
				\caption{$A_1=\mathbb{D}\cap B_{q_1}(2r)$.}
				\label{fig:Bola q12r.png}
			\end{figure}	
			Observe that, applying $|\nabla d(x,z)^2|\leq 2d(x,z)$, we have
			
			\begin{equation}\label{eq:gradientineq}
				|\nabla\varphi_{\mu,q_i}(x)|=2\mu^2\frac{|\nabla d^2(x,q_i)|}{\mu^2d^2(x,q_i)-1}\leq4\mu^2\frac{d(x,q_i)}{\mu^2d^2(x,q_i)-1},\quad\text{for every }x\in\mathbb{D}.
			\end{equation}
			Then, for any $x\in\mathbb{D}\setminus (\cup_{i=1}^{k}A_i)$ one has
			
			\begin{equation}\label{eq: inequalitiesDminus}
				\left|\nabla\varphi_{\mu,q_i}\right|\leq\frac{8r\mu^2}{4\mu^2r^2-1},\quad\left|e^{\varphi_{\mu,q_i}}\right|\leq\frac{4\mu^2}{(4\mu^2r^2-1)^2}.
			\end{equation}
			Then, using the fact that for any $i,j=1,\ldots,k$, such that $i\neq j$, we have
			
			\begin{equation*}
				\int_{A_i}\left|\nabla\varPhi_{\mu,k}\right|^2=\int_{A_j}\left|\nabla\varPhi_{\mu,k}\right|^2,
			\end{equation*}
			the equation (\ref{eq: SplitGradientofSum}) can be controlled by
			
			\begin{equation*}
				\int_{\mathbb{D}}|\nabla\varPhi_{\mu,k}|^2\leq k\int_{A_1}|\nabla\varPhi_{\mu,k}|^2+O(1).
			\end{equation*}
			
			Now, using \eqref{eq: GradientPhi} we have
			
			\begin{equation*}
				\begin{split}
					\int_{A_1}|\nabla\varPhi_{\mu,k}|^2&=\int_{A_1}\left|\frac{\nabla\varphi_{\mu,q_1}e^{\varphi_{\mu,q_1}}}{\sum_{i=1}^{k}e^{\varphi_{\mu,q_i}}}\right|^2+\int_{A_1}\left|\frac{\sum_{i=2}^{k}\nabla\varphi_{\mu,q_i}e^{\varphi_{\mu,q_i}}}{\sum_{i=1}^{k}e^{\varphi_{\mu,q_i}}}\right|^2\\
					&+2\int_{A_1}\frac{\left|\nabla\varphi_{\mu,q_1}e^{\varphi_{\mu,q_1}}\sum_{i=2}^{k}\nabla\varphi_{\mu,q_i}e^{\varphi_{\mu,q_i}}\right|}{\left|\sum_{i=1}^{k}e^{\varphi_{\mu,q_i}}\right|^2}\\
					&\leq I_1+I_2+O(1),
				\end{split}
			\end{equation*}
			where $I_1$ and $I_2$ are defined as
			
			\begin{equation*}
				I_1:=\int_{A_1}|\nabla\varphi_{\mu,q_1}|^2\quad\text{  and }\quad I_2:=\int_{A_1}\frac{\left|\nabla\varphi_{\mu,q_1}e^{\varphi_{\mu,q_1}}\sum_{i=2}^{k}\nabla\varphi_{\mu,q_i}e^{\varphi_{\mu,q_i}}\right|}{\left|\sum_{i=1}^{k}e^{\varphi_{\mu,q_i}}\right|^2}.
			\end{equation*}
			
			Note that, by Hölder inequality and (\ref{eq: inequalitiesDminus}), we can control the second integral by $|\nabla\varphi_{\mu,q}|$, thus
			
			\begin{equation*}
				\begin{split}
					I_2:=&\int_{A_1}\frac{\left|\nabla\varphi_{\mu,q_1}e^{\varphi_{\mu,q_1}}\sum_{i=2}^{k}\nabla\varphi_{\mu,q_i}e^{\varphi_{\mu,q_i}}\right|}{\left|\sum_{i=1}^{k}e^{\varphi_{\mu,q_i}}\right|^2}\\
					&\leq\left(\int_{A_1}\left|\frac{\nabla\varphi_{\mu,q_1}e^{\varphi_{\mu,q_1}}}{\sum_{i=1}^{k}e^{\varphi_{\mu,q_i}}}\right|^2\right)^{\frac{1}{2}}\left(\int_{A_1}\left|\frac{\sum_{i=2}^{k}\nabla\varphi_{\mu,q_i}e^{\varphi_{\mu,q_i}}}{\sum_{i=1}^{k}e^{\varphi_{\mu,q_i}}}\right|^2\right)^{\frac{1}{2}}\\
					&\leq\left(\int_{A_1}\left|\nabla\varphi_{\mu,q_1}\right|^2\right)^{\frac{1}{2}}+O(1).
				\end{split}
			\end{equation*}
			
			Thus, the following estimate for the equation (\ref{eq: SplitGradientofSum}) is obtained
			
			\begin{equation*}
				\int_{\mathbb{D}}|\nabla\varPhi_{\mu,k}|^2\leq k\left(I_1+\sqrt{I_1}\right)+O(1).
			\end{equation*}
			
			Our next step is to quantize the right-hand side of the equation. To do this, suppose that $q_1$ is close enough to $p_1$, such that $d^2(x,q_1)=r^2+o_{r}(1)$, and for simplicity assume $p_1=(-1,0)$, then one can find that
			
			\begin{equation*}
				\begin{split}
					I_1&\leq16\mu^4\int_{A_1}\frac{d^2(x,q_1)}{\left(\mu^2d^2(x,q_1)-1\right)^2}\\
					&\leq16\mu^4\int_{A^{+}}\frac{r^2+o_{r}(1)}{\left(\mu^2d^2(x,q_1)-1\right)^2},
				\end{split}
			\end{equation*}
			where 
			
			\begin{equation*}
				A^{+}=\left\{(x_1,x_2)\in\mathbb{R}^2:\left(x_1+1+r\right)^2+x_2^2\leq4r^2,x_1\geq-1\right\}.    
			\end{equation*}
			
			\begin{figure}[ht]
				\centering
				\includegraphics[width=9.5cm]{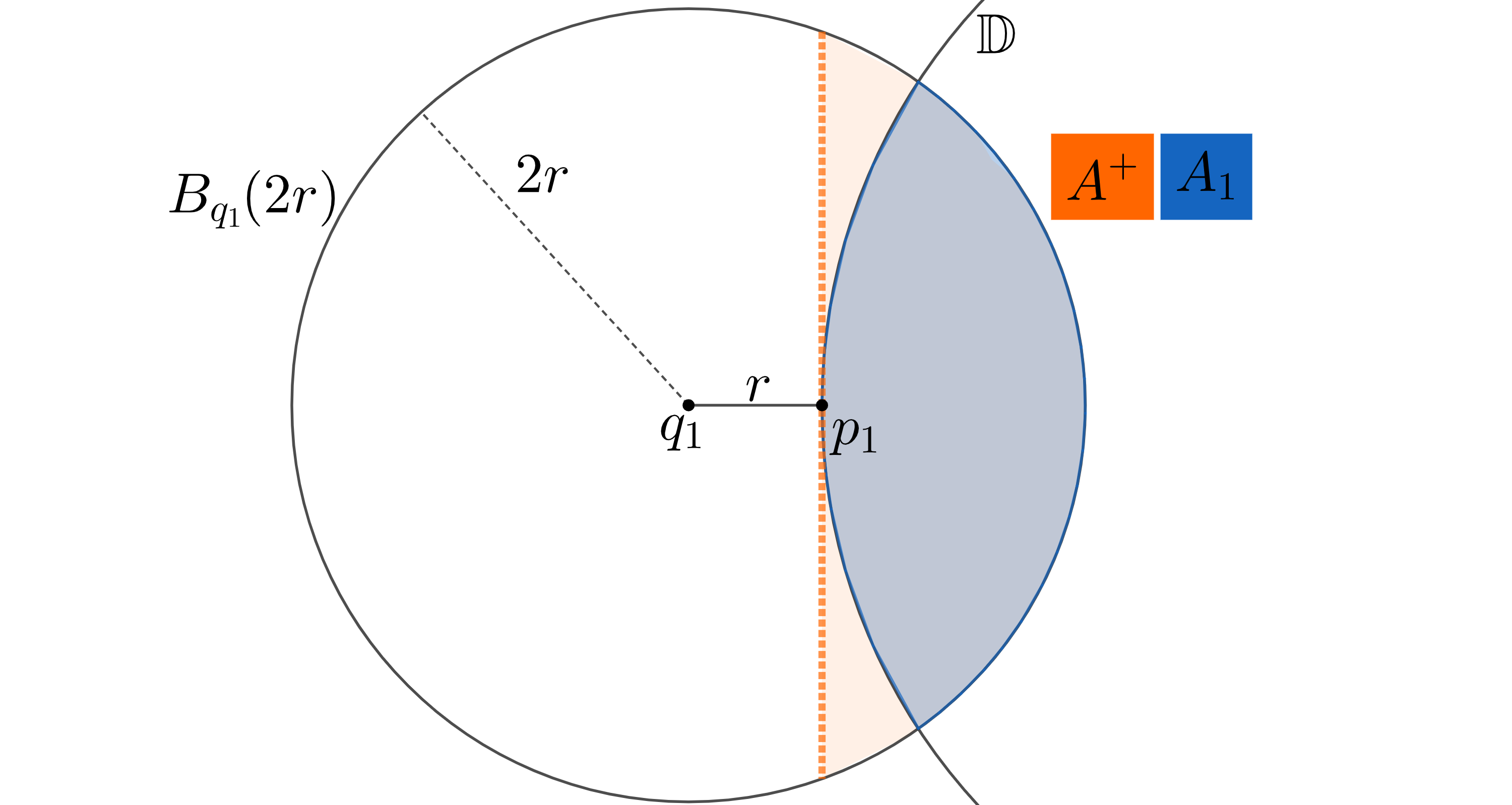}
				\caption{$A^+$ and $A_1$.}
				\label{fig:A+ y A1.png}
			\end{figure}

			Now, taking polar coordinates ($\theta$, $\rho$) with center at $q_1$ and considering the change of variable $t=\mu^2\rho^2-1$, we have
			
			\begin{equation*}
				\begin{split}
					(A):=\int_{A^{+}}\frac{1}{\left(\mu^2d^2(x,q_1)-1\right)^2}&=2\int_{0}^{\frac{\pi}{3}}\int_{\frac{r}{\cos{\theta}}}^{2r}\frac{\rho}{\left(\mu^2\rho^2-1\right)^2}d\rho d\theta\\
					=\frac{1}{\mu^2}\int_{0}^{\frac{\pi}{3}}\int_{\frac{\mu^2r^2}{\cos^2{\theta}}-1}^{4\mu^2r^2-1}\frac{1}{t^2}dt d\theta&=\frac{1}{\mu^2}\int_{0}^{\frac{\pi}{3}}\frac{\cos^2{\theta}}{\mu^2r^2-\cos^2{\theta}} d\theta+O(1)\\
					=\frac{\left[\mu r \arctan{\frac{\mu r\tan{\theta}}{\sqrt{\mu^2 r^2-1}}}\right]_{0}^{\frac{\pi}{3}}}{\mu^2\sqrt{\mu^2 r^2-1}}+O(1)&=\frac{\pi r}{2\mu\sqrt{\mu^2 r^2-1}}+O(1).
				\end{split}
			\end{equation*}
			
			So we can conclude that
			
			\begin{equation*}
				I_1\leq\frac{8\pi\mu^3r^3}{\sqrt{\mu^2 r^2-1}}+o\left(\frac{1}{\sqrt{\mu^2 r^2-1}}\right),
			\end{equation*}
			and, therefore
			
			\begin{equation*}
				I_2\leq\frac{2\mu r\sqrt{2\pi\mu r}}{\sqrt[4]{\mu^2 r^2-1}}+o\left(\frac{1}{\sqrt[4]{\mu^2 r^2-1}}\right).
			\end{equation*}
			And with this we finish the proof of (\ref{eq: GradientofLog}).\\
			
			\underline{Estimate of (\ref{eq: ExpofLog}):}
			As before, we decompose the expression of (\ref{eq: ExpofLog}) by
			
			\begin{equation}\label{eq: EplitExpofPhis}
				\int_{\mathbb{D}}e^{\varPhi_{\mu,k}}=k\int_{A_1}e^{\varPhi_{\mu,k}}+\int_{\mathbb{D}\setminus(\cup_{i=1}^{k}A_i)}e^{\varPhi_{\mu,k}}.
			\end{equation}
			
			Note that, using the inequalities collected in (\ref{eq: inequalitiesDminus}), we can control the integral at $\mathbb{D}\setminus(\cup_{i=1}^{k}A_i)$ so it is sufficient that we focus on the one in $A_1$. Then, by definition (\ref{eq: LogSumgofBubbles}) we have
			
			\begin{equation*}
				\int_{A_1}e^{\varPhi_{\mu,k}}=\int_{A_1}e^{\varphi_{\mu,q_1}}+\int_{A_1}\sum_{i=2}^ke^{\varphi_{\mu,q_i}}\leq\int_{A_1}e^{\varphi_{\mu,q_1}}+O(1),
			\end{equation*}    
			
			By simple computations we can find that	
			
			\begin{equation*}
				\int_{A_1}e^{\varphi_{\mu,q_1}}=2\int_{A^+}\frac{4\mu^2 dx}{(\mu^2d(x,q_1)^2-1)^2}=8\mu^2\cdot (A)=\frac{2\pi\mu r}{\sqrt{\mu^2 r^2-1}}+O(1),
			\end{equation*}    
			
			and (\ref{eq: ExpofLog}) comes from $(\ref{eq: EplitExpofPhis})$.\\
			
			\underline{Estimate of (\ref{eq: ExpofLogBoundary}):}
			Following the same strategy as before first we divide the integral (\ref{eq: ExpofLogBoundary}) as
			
			\begin{equation*}   
				\int_{\partial\mathbb{D}}\mathfrak{D}e^{\frac{\varPhi_{\mu,k}}{2}}=k\int_{A_1}\mathfrak{D}e^{\frac{\varPhi_{\mu,k}}{2}}+\int_{\mathbb{D}\setminus(\cup_{i=1}^{k}A_i)}\mathfrak{D}e^{\frac{\varPhi_{\mu,k}}{2}}.
			\end{equation*}
			where $A_i=\partial\mathbb{D}\cap B_{p_i}(r)$. 
			\begin{figure}[ht]
				\centering
				\includegraphics[width=9.5cm]{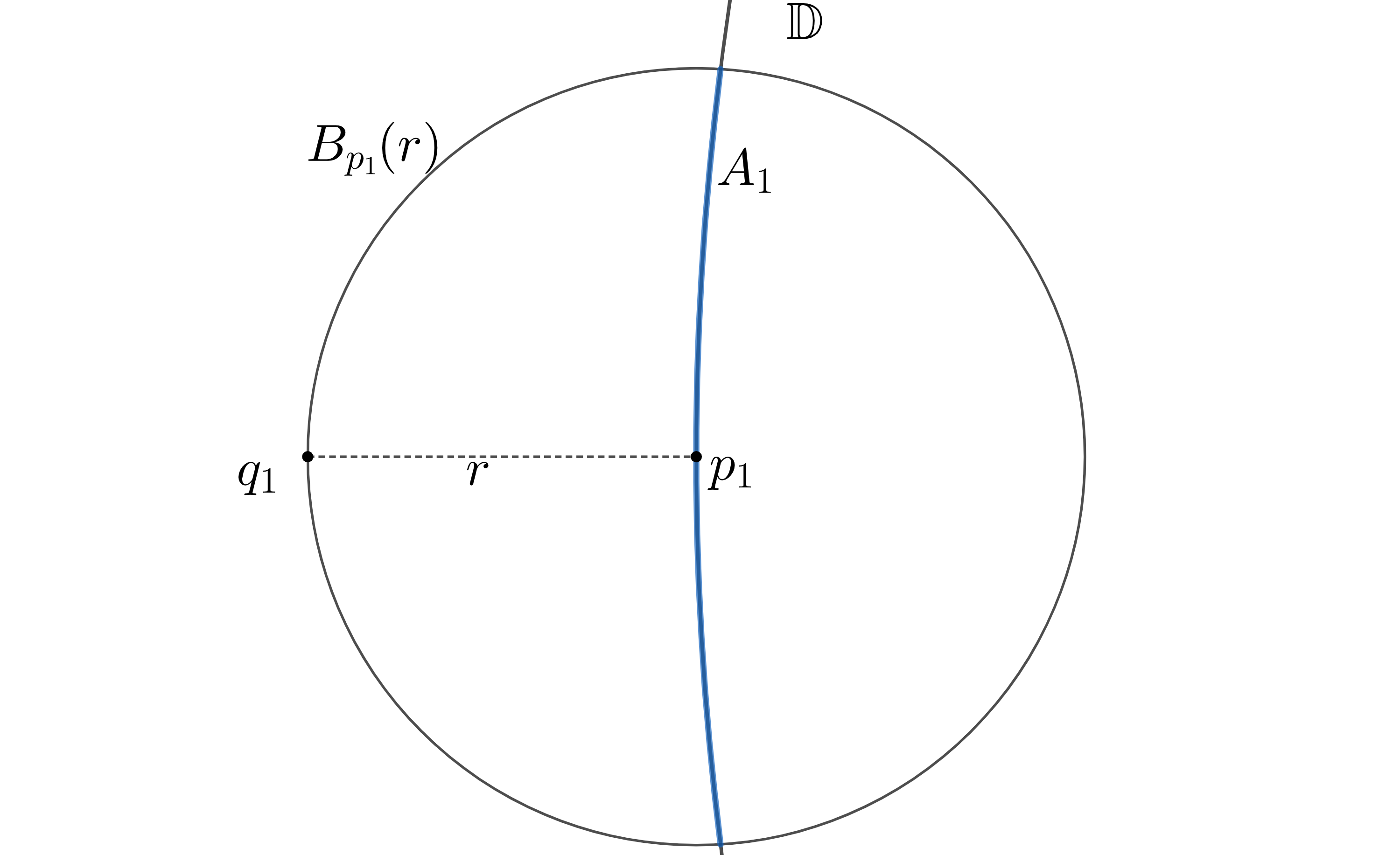}
				\caption{$A_1=\partial\mathbb{D}\cap B_{p_1}(r)$.}
				\label{fig:Bola P1r.png}
			\end{figure}
			
			Now, since we can control from below each integral by the minimum of $\mathfrak{D}$ and the inequalities (\ref{eq: inequalitiesDminus}), we have
			
			\begin{equation*}
				\begin{split}   
					\int_{\partial\mathbb{D}}\mathfrak{D}e^{\frac{\varPhi_{\mu,k}}{2}}&\geq k\min_{x\in A_1}\mathfrak{D}(x)\int_{A_1}e^{\frac{\varPhi_{\mu,k}}{2}}+\min_{x\in \mathbb{D}\setminus(\cup_{i=1}^{k}A_i)}\mathfrak{D}(x)\int_{\mathbb{D}\setminus(\cup_{i=1}^{k}A_i)}e^{\frac{\varPhi_{\mu,k}}{2}}\\
					&\geq k\min_{x\in A_1}\mathfrak{D}(x)\int_{A_1}e^{\frac{\varPhi_{\mu,k}}{2}}+O(1).
				\end{split}
			\end{equation*}
			And  the integral on $A_1$ can be controlled by
			
			\begin{equation*}
				\int_{A_1}e^{\frac{\varPhi_{\mu,k}}{2}}=\int_{A_1}\left(\sum_{i=1}^ke^{\varphi_{\mu,q_i}}\right)^{1/2}\geq\int_{A_1}e^{\frac{\varphi_{\mu,q_1}}{2}}.
			\end{equation*}
			Finally (\ref{eq: ExpofLogBoundary}) comes from the following computation.
			
			\begin{equation*}
				\int_{A_1}e^{\frac{\varphi_{\mu,q_1}}{2}}=2\int_{A^+}\frac{2\mu dx}{\mu^2d(x,q_1)^2-1}=2\oint_{0}^{r}\frac{2\mu dx_2}{\mu^2(r^2+x_2^2)-1}=\frac{2\pi}{\sqrt{\mu^2 r^2-1}}+O(1).
			\end{equation*}
			
			\underline{Estimate of (\ref{eq: LogofPhis}):}
			First, note that
			
			\begin{equation*}
				\int_{\partial\mathbb{D}}\tilde{\varPhi}_{\mu,k}\leq \int_{\partial\mathbb{D}}\varPhi_{\mu,k}+O(1). 
			\end{equation*}
			Now, if we split the right hand of the last inequality as before, we have
			
			\begin{equation*}
				\begin{split}
					\int_{\partial\mathbb{D}}\varPhi_{\mu,k}&=k\int_{A_1}\varPhi_{\mu,k}+\int_{\mathbb{D}\setminus(\cup_{i=1}^{k}A_i)}\varPhi_{\mu,k}\\    		
					&\leq k\int_{A_1}\varPhi_{\mu,k}+O(1). 
				\end{split}
			\end{equation*}
Next, since for any $x\in\partial\mathbb{D}\setminus A_1$ one have that $d(x,q_1)<d(x,q_i)$ for $2\leq i<k$, then
			
			\begin{equation*}
				\begin{split}
					\int_{A_1}\varPhi_{\mu,k}&=\int_{A_1}\log{\left(\sum_{i=1}^ke^{\varphi_{\mu,q_i}}\right)}\\    		
					&\leq \int_{A_1}\log{\left(ke^{\varphi_{\mu,q_1}}\right)}\\
					&\leq \int_{A_1}\log{\left(e^{\varphi_{\mu,q_1}}\right)}+O(1).
				\end{split}
			\end{equation*}
			
			Then, we can estimate
			
			\begin{equation*}
				\begin{split}
					\int_{A_1}\varphi_{\mu,q_1}&=4\int_{0}^{r}\log\frac{2\mu}{\mu^2(r^2+x_2^2)-1}dx_2\\
					&=4\left[x_2\log\frac{2\mu}{\mu^2(r^2+x_2^2)-1}\right]_{0}^{r}+O(1)=O(1).
				\end{split}
			\end{equation*}
			This concludes (\ref{eq: LogofPhis}).
		\end{proof}
	\end{lemma}
	
	\textbf{Case 2:} $\mathbf{G=O(2).}$ Given a parameter $\mu>1$, for every $x\in\mathbb{D}$ we define the functions
	\begin{equation}\label{eq: radialBubbles}
		\begin{array}{cc}
			\varphi_{\mu}:\mathbb{D}\rightarrow\mathbb{R}, & \varphi_{\mu}(x)=2\log\frac{2\mu}{\mu^2-|x|^2},\\ \\

			\tilde{\varphi}_{\mu}:\mathbb{D}\rightarrow\mathbb{R}, & \tilde{\varphi}_{\mu}(x)=\varphi_{\mu}(x)-\log|K(x)|.
		\end{array}
	\end{equation}
	Note that if $\mu$ tends to 1, then the function $\tilde{\varphi}_{\mu}$ blows up on the entire boundary of the disk. In particular, the blow-up set is infinite here.
	\begin{lemma}\label{lem:testradial}
		Under the radial assumptions and $\mathfrak{D}(x)=\tilde{\mathfrak{D}}>1$, let $\mathcal{I}$ defined in \eqref{eq: Energy Functional Disk}, then there exists a function $\tilde{\varphi}_{\mu}$ defined in \eqref{eq: radialBubbles} such that
		
		\begin{equation*}
			\label{eq: radialbubble}
			\mathcal{I}\left(\tilde{\varphi}_{\mu}\right)\rightarrow-\infty,\quad\int_{\partial\mathbb{D}}e^{\tilde{\varphi}_{\mu}/2}\rightarrow+\infty\quad as\quad\mu\searrow1.
		\end{equation*}
		\begin{proof}
			Let us remember the energy functional we are working with.
			
			\begin{equation*}
				\mathcal{I}\left(\tilde{\varphi}_{\mu}\right)=\int_{\mathbb{D}}\left(\frac{1}{2}|\nabla \tilde{\varphi}_{\mu}|^2+2e^{{\varphi}_{\mu}}\right)+\int_{\partial\mathbb{D}}\left(2\tilde{\varphi}_{\mu}-4\tilde{\mathfrak{D}}\sqrt{e^{{\varphi}_{\mu}}}\right).
			\end{equation*}
			
			As in the proof of the previous lemma, we will now estimate each of the parts of the above expression. First of all, note that we can control the gradient as in \eqref{eq: gradientestimate} by
			
			\begin{equation*}
				\frac{1}{2}\int_{\mathbb{D}}|\nabla\tilde{\varphi}_{\mu}|^2\leq\left(\frac{1}{2}+\varepsilon\right)\int_{\mathbb{D}}|\nabla \varphi_{\mu}|^2+O(1),
			\end{equation*}
			where,
			
			\begin{equation*}
				\int_{\mathbb{D}}|\nabla\varphi_{\mu}|^2=\int_{\mathbb{D}}\frac{16dx}{(\mu^2-|x|^2)^2}=32\pi\int_{0}^{1}\frac{\rho d\rho}{(\mu^2-\rho^2)^2}=32\pi\cdot(C).
			\end{equation*}
			Moreover, the other integral defined inside the domain is equal to
			
			\begin{equation*}
				\int_{\mathbb{D}}e^{\varphi_{\mu}}=\int_{\mathbb{D}}\frac{4\mu^2dx}{(\mu^2-|x|^2)^2}=8\pi\mu^2\pi\int_{0}^{1}\frac{\rho d\rho}{(\mu^2-\rho^2)^2}=8\pi\mu^2\cdot(C).
			\end{equation*}
			Therefore, to study the first integral of the energy functional it is sufficient to compute $(C)$. Thus
			
			\begin{equation*}
				(C)=\int_{0}^{1}\frac{\rho d\rho}{(\mu^2-\rho^2)^2}=\frac{1}{2}\left[\frac{1}{\mu^2-\rho^2}\right]_{\rho=0}^{\rho=1}=\frac{1}{2\mu^2(\mu^2-1)}.
			\end{equation*}
			On the other hand, we have the boundary terms that can be controlled as follows
			
			\begin{equation*}
				\begin{split}
					\int_{\partial\mathbb{D}}\tilde{\mathfrak{D}}e^{\frac{\varphi_{\mu}}{2}}&=\tilde{\mathfrak{D}}\int_{\partial\mathbb{D}}\frac{2\mu}{\mu^2-|x|^2}dx=\tilde{\mathfrak{D}}\frac{4\pi \mu}{\mu^2-1},
				\end{split}	
			\end{equation*}
			and,
			
			\begin{equation*}
				\int_{\partial\mathbb{D}}\varphi_{\mu}=\int_{\partial\mathbb{D}}2\log\frac{2\mu}{\mu^2-|x|^2}=4\pi\log\frac{2\mu}{\mu^2-1}\leq o\left(\frac{1}{\mu^2-1}\right).
			\end{equation*}
			In summary, the following estimates hold:
			
			\begin{equation*}
				\int_{\mathbb{D}}|\nabla\varphi_{\mu}|^2\leq\frac{16\pi}{\mu^2-1}+o\left(\frac{1}{\mu^2-1}\right),
			\end{equation*}
			
			\begin{equation*}
				\int_{\mathbb{D}}e^{\varphi_{\mu}}\leq\frac{4\pi}{\mu^2-1}+o\left(\frac{1}{\mu^2-1}\right),
			\end{equation*}
			
			\begin{equation*}
				\int_{\partial\mathbb{D}}\mathfrak{D}e^{\frac{\varphi_{\mu}}{2}}= \tilde{\mathfrak{D}}\frac{4\pi}{\mu^2-1}+O(1),
			\end{equation*}
			
			\begin{equation*}
				\int_{\partial\mathbb{D}}\varphi_{\mu}\leq o\left(\frac{1}{\mu^2-1}\right).
			\end{equation*}
			And, since $\tilde{\mathfrak{D}}>1$, as $\mu\searrow 1$, we can conclude that
			
			\begin{equation*}
				\mathcal{I}\left(\tilde{\varphi}_{\mu}\right)\leq\frac{16\pi}{\mu^2-1}\left(1-\tilde{\mathfrak{D}}\right)\to-\infty\quad\text{ and }\quad\int_{\partial\mathbb{D}}e^{\tilde{\varphi}_{\mu}/2}\rightarrow+\infty,
			\end{equation*}
			as claimed.
		\end{proof}
	\end{lemma}	
	\bibliographystyle{plain}
	\bibliography{references.bib} 
\end{document}